\documentclass{amsproc}

\usepackage{amsmath}
\usepackage{amssymb}
\usepackage{amsxtra}
\usepackage{latexsym}
\usepackage{ifthen}

\usepackage[dvips]{graphicx}
\usepackage{epstopdf}
\usepackage{epsfig}
\usepackage{epic}
\usepackage{eepic}

\def\Xint#1{\mathchoice
   {\XXint\displaystyle\textstyle{#1}}
   {\XXint\textstyle\scriptstyle{#1}}
   {\XXint\scriptstyle\scriptscriptstyle{#1}}
   {\XXint\scriptscriptstyle\scriptscriptstyle{#1}}
   \!\int}
\def\XXint#1#2#3{{\setbox0=\hbox{$#1{#2#3}{\int}$}
     \vcenter{\hbox{$#2#3$}}\kern-.5\wd0}}
\def\dashint{\Xint-}

\newtheorem{theorem}{Theorem}[section]
\newtheorem{lemma}[theorem]{Lemma}
\newtheorem{corollary}[theorem]{Corollary}
\newtheorem{propo}[theorem]{Proposition}

\theoremstyle{definition}

\theoremstyle{remark}
\newtheorem{remark}[theorem]{Remark}

\numberwithin{equation}{section}

\def\Xint#1{\mathchoice
   {\XXint\displaystyle\textstyle{#1}}%
   {\XXint\textstyle\scriptstyle{#1}}%
   {\XXint\scriptstyle\scriptscriptstyle{#1}}%
   {\XXint\scriptscriptstyle\scriptscriptstyle{#1}}%
   \!\int}
\def\XXint#1#2#3{{\setbox0=\hbox{$#1{#2#3}{\int}$}
     \vcenter{\hbox{$#2#3$}}\kern-.5\wd0}}
\def\dashint{\Xint-}

\begin{document}

\title[On the Dirichlet problem for degenerate Beltrami equations]{On the Dirichlet problem for degenerate Beltrami equations}

\author{Vladimir Ryazanov}
\address{Institute of Applied Mathematics and Mechanics, National Academy of Sciences of Ukraine,
 Donetsk, Ukraine}
\email{vlryazanov1@rambler.ru , vl.ryazanov1@gmail.com}

\author{Ruslan Salimov}
\address{Institute of Applied Mathematics and Mechanics, National Academy of Sciences of Ukraine, Donetsk, Ukraine}
\email{ruslan623@yandex.ru , salimov07@rambler.ru}

\author{Uri Srebro}
\address{Technion, Haifa, Israel}
\email{srebro@math.technion.ac.il , srebro@techunix.technion.ac.il}

\author{Eduard Yakubov}
\address{Holon Institute of Technology, Holon, Israel}
\email{yakubov@hit.ac.il , eduardyakubov@gmail.com}

\date{\today}



\begin{abstract}

We show that every homeomorphic $W^{1,1}_{\rm loc}$ solution $f$ to
a Beltrami equation $\overline{\partial}f=\mu\,\partial f$ in a
domain $D\subset\Bbb C$ is the so--called lower
$Q-$homeo\-mor\-phism with $Q(z)=K^T_{\mu}(z, z_0)$ where
$K^T_{\mu}(z, z_0)$ is the tangent dilatation of $f$ with respect to
an arbitrary point $z_0\in {\overline{D}}$ and develop the theory of
the boundary behavior of such solutions. Then, on this basis, we
show that, for wide classes of degenerate Beltrami equations
$\overline{\partial}f=\mu\,\partial f$, there exist regular
solutions of the Dirichlet problem in arbitrary Jordan domains in
$\Bbb C$ and pseudoregular and multi-valued solutions in arbitrary
finitely  connected domains in $\Bbb C$ bounded by mutually disjoint
Jordan curves.
\end{abstract}

\maketitle

\tableofcontents

\section{Introduction}

Let $D$ be a domain in the complex plane ${\Bbb C}$, i.e., a
connected open subset of ${\Bbb C}$, and let $\mu:D\to{\Bbb C}$ be
a measurable function with $|\mu(z)|<1$ a.e. (almost everywhere) in
$D$. A {\bf Beltrami equation} is an equation of the form
\begin{equation}\label{eqBeltrami} f_{\bar z}=\mu(z)\,f_z\end{equation} where
$f_{\bar z}=\overline{\partial}f=(f_x+if_y)/2$,
$f_{z}=\partial f=(f_x-if_y)/2$, $z=x+iy$, and $f_x$ and $f_y$ are
partial derivatives of $f$ in $x$ and $y$, correspondingly. The
function $\mu$ is called the {\bf complex coefficient} and
\begin{equation}\label{eqKPRS1.1}K_{\mu}(z)=\frac{1+|\mu(z)|}{1-|\mu(z)|}\end{equation}
the {\bf dilatation quotient} of the equation (\ref{eqBeltrami}).
The Beltrami equation (\ref{eqBeltrami}) is said to be {\bf
degenerate} if ${\rm ess}\,{\rm sup}\,K_{\mu}(z)=\infty$. The
existence of homeomorphic $W^{1,1}_{\rm loc}$ solutions was recently
established to many degenerate Beltrami equations, see, e.g.,
related references in the recent monographs \cite{GRSY^*} and
\cite{MRSY} and in the surveys \cite{GRSY} and \cite{SY}.

Given a point $z_0$ in $\Bbb C$, we also apply here the quantity
\begin{equation}\label{eqTangent} K^T_{\mu}(z,z_0)\ =\
\frac{\left|1-\frac{\overline{z-z_0}}{z-z_0}\mu (z)\right|^2}{1-|\mu (z)|^2} \end{equation} that is called the
{\bf tangent dilatation} of the Beltrami equation (\ref{eqBeltrami}) with respect to $z_0$, see, e.g.,
\cite{RSY$_3$}, cf. the corresponding terms and notations in \cite{And3, GMSV$_2$} and \cite{Le}. Note that
\begin{equation}\label{eqConnect} K^{-1}_{\mu}(z)\leqslant K^T_{\mu}(z,z_0) \leqslant K_{\mu}(z)
\end{equation} for all $z_0\in \Bbb C$ and $z\in D$.

Let us clear a geometric sense of the tangent dilatation. A point
$z\in\Bbb C$ is called a \textbf{regular point} for a mapping
$f:D\to\Bbb C$ if  $f$ is differentiable at $z$ and $J_{f}(z)\neq
0.$ Given ${\omega
}\in\mathbb{C},$ $|{\omega }|=1,$ the \textbf{derivative in the direction} ${%
\omega }$ of the mapping $f$ at the point $z$ is
\begin{equation}  \label{Ring_1.7A}
\partial_\omega f(z)\ =\ \lim\limits_{t\to +0}\ \frac{f(z+t\cdot {\omega })-f(z)%
}{t}\ .
\end{equation}

The \textbf{radial direction} at a point $z\in D$ with respect to
the center $z_0\in \mathbb{C},$ $z_0\ne z,$ is
\begin{equation}  \label{Ring_eq3.3}
{\omega }_0\ =\ {\omega }_0(z,z_0)\ =\ \frac{z-z_0}{|z-z_0|}\ .
\end{equation} The \textbf{tangent direction} at a point $z\in D$ with respect to
the center $z_0\in \mathbb{C},$ $z_0\ne z,$ is ${\tau } =i{\omega }%
_0 .$ The \textbf{tangent dilatation} of $f$ at $z$ with respect to
$z_0$ is the quantity
\begin{equation}  \label{Ring_eq3.2y}
K^T(z, z_0, f)\ \colon =\ \frac{|\partial_T^{z_0}
f(z)|^2}{|J_f(z)|},
\end{equation}
where $\partial_T^{z_0} f(z)$ is the derivative of $f$ at $z$ in the
direction ${\tau }$.

Note that if $z$ is a regular point of $f$ and $|{\mu }(z)|<1,$ ${\mu }%
(z)=f_{{\overline {{z}}}}/f_z,$ then
\begin{equation}  \label{Ring_eq1.5y}
K^T(z,z_0,f)\ =\ K^T_{{\mu }}(z,z_0),
\end{equation}
i.e.
\begin{equation}  \label{Ring_eq1.5P}
K^T_{{\mu }}(z,z_0)\ =\ \frac{|\partial_T^{z_0} f(z)|^2}{|J_f(z)|}\
.
\end{equation}
Indeed, the equalities (\ref{Ring_eq1.5y}) and (\ref{Ring_eq1.5P})
follow directly from the calculations
\begin{equation}  \label{Ring_eq1.5osm}
\partial^{z_0}_T f=\frac{1}{r}\left(\frac{\partial f }{\partial z}\cdot \frac{%
\partial z}{\partial \vartheta}+\frac{\partial f }{\partial {\overline {{z}}}%
}\cdot \frac{\partial {\overline {{z}}} }{\partial \vartheta}\right)
=
i\cdot \left(\frac{z-z_0}{|z-z_0|}\cdot f_z-\frac{\overline{z-z_0}} {|z-z_0|}%
\cdot f_{{\overline {{z}}}}\right)
\end{equation}
where $r=|z-z_0|$ and $\vartheta = \mbox{arg}\, (z-z_ 0)$ because $J_f(z)=|f_{z}|^2-|f_{{%
\overline {{z}}}}|^2.$

\medskip

Recall that every holomorphic (analytic) function $f$ in a domain
$D$ in ${\Bbb C}$ satisfies the simplest Beltrami equation
\begin{equation}\label{eqKPRS1.2}f_{\bar z}=0\end{equation} with
$\mu(z)\equiv0$. If a holomorphic function $f$ given in the unit
disk ${\Bbb D}=\{z\in \mathbb{C}: |z|<1 \}$ is continuous in its closure, then by the Schwarz
formula \begin{equation}\label{eqKPRS1.3}f(z)=i\,{\rm
Im}\,f(0)+\frac{1}{2\pi i}\int\limits_{|\zeta|=1}{\rm
Re}\,f(\zeta)\cdot\frac{\zeta+z}{\zeta-z}\frac{d\zeta}{\zeta}\,,\end{equation}
see, e.g., Section 8, Chapter III, Part 3 in \cite{HuCo}. Thus,
the holomorphic function $f$ in the unit disk ${\Bbb D}$ is
determinated, up to a purely imaginary additive constant $ic$,
$c={\rm Im}\,f(0)$, by its real part $\varphi(\zeta)={\rm
Re}\,f(\zeta)$ on the boundary of ${\Bbb D}$.

\medskip

Hence the {\bf Dirichlet problem} for the Beltrami equation
(\ref{eqBeltrami}) in a domain $D\subset{\Bbb C}$ is the problem on
the existence of a continuous function $f:D\to{\Bbb C}$ having
partial derivatives of the first order a.e., satisfying
(\ref{eqBeltrami}) a.e. and such that
\begin{equation}\label{eqGrUsl}\lim\limits_{z\to\zeta}{\rm
Re}\,f(z)=\varphi(\zeta)\qquad\forall\ \zeta\in\partial D\end{equation} for a prescribed continuous function
$\varphi:\partial D\to{\Bbb R},$ see, e.g., \cite{Boj} and \cite{Vekua}. It is obvious that if $f$ is a solution
of this problem, then the function $F(z)=f(z)+ic$, $c\in{\Bbb R}$, is also  so.

Recall that a mapping $f:D\to{\Bbb C}$ is called {\bf discrete} if
the preimage  $f^{-1}(y)$ consists of isolated points for every
$y\in{\Bbb C}$, and {\bf open} if $f$ maps every open set
$U\subseteq D$ onto an open set in ${\Bbb C}$.

If $\varphi(\zeta)\not\equiv{\rm const}$, then the {\bf regular
solution}  of the Dirichlet problem (\ref{eqGrUsl}) for the Beltrami
equation (\ref{eqBeltrami}) is a continuous, discrete and open
mapping $f:D\to{\Bbb C}$ of the Sobolev class $W_{\rm loc}^{1,1}$
with its Jacobian $J_f(z)=|f_z|^2-|f_{\bar z}|^2\neq0$ a.e.
satisfying (\ref{eqBeltrami}) a.e. and the condition
(\ref{eqGrUsl}). The regular solution of such a problem with
$\varphi(\zeta)\equiv c$, $\zeta\in\partial D$, for the Beltrami
equation (\ref{eqBeltrami}) is the function $f(z)\equiv c$, $z\in
D$.

Examples given in the paper \cite{Dy} show that, even in the case of the simplest domain, the unit disk ${\Bbb
D}$ in ${\Bbb C}$, any power of the integrability of the dilatation $K_{\mu}$ does not guarantee the existence
of the regular solutions of the Dirichlet problem (\ref{eqGrUsl}) for the Beltrami equation (\ref{eqBeltrami})
if $\varphi(\zeta)\not\equiv{\rm const}$. The corresponding criteria has a much more complicated nature.

\medskip

Boundary value problems for the Beltrami equations are due to the
well-known Riemann dissertation in the case of $\mu(z)=0$ and to the
papers of Hilbert (1904, 1924) and Poincare (1910) for the
corresponding Cauchy--Riemann system. The Dirichlet problem was well
studied for uniformly elliptic systems, see, e.g., \cite{Boj} and
\cite{Vekua}. The Dirichlet problem for degenerate Beltrami
equations in the unit disk was recently studied in \cite{Dy}. In
comparison with this work, our approach is based on estimates of the
modulus of dashed lines but not of paths under arbitrary
homeomorphic $W^{1,1}_{\rm loc}$ solutions of the Beltrami
equations.

\medskip
Recently in \cite{KPR}, it was showed that every homeomorphic
$W^{1,1}_{\rm loc}$ solution $f$ to a Beltrami equation
(\ref{eqBeltrami}) in a domain $D\subset\Bbb C$ is the so--called
lower $Q-$ho\-meo\-mor\-phism with $Q(z)=K_{\mu}(z)$ at an arbitrary
point $z_0\in {\overline{D}}$, and in \cite{KPR1}, it was formulated
new exixtence theorems for the Dirichlet problem to the Beltrami
equations in terms of $K_{\mu}(z)$. Here we show that $f$ is the
lower $Q-$homeomorphism with $Q(z)=K^T_{\mu}(z, z_0)$ at each point
$z_0\in {\overline{D}}$, see further Theorem \ref{thDIR1}. This is
the basis for developing the theory of the boundary behavior of
solutions that can be applied in tern to the research of various
boundary value problems for (\ref{eqBeltrami}). Namely, we prove,
for wide classes of degenerate Beltrami equations
(\ref{eqBeltrami}), that the Dirichlet problem (\ref{eqGrUsl}) has
regular solutions in an arbitrary Jordan domain and pseudoregular
and multi-valued solutions in an arbitrary finitely connected domain
bounded by a finite collection of mutually disjoint Jordan curves.
The main criteria are formulated by us in terms of the tangent
dilatations $K^T_{\mu}(z,z_0)$ which are more refined although the
corresponding criteria remain valid, in view of (\ref{eqConnect}),
for the usual dilatation $K_{\mu}(z)$, too.

\medskip

Throughout this paper, $B(z_0,r)=\{z\in{\Bbb C}:|z_0-z|<r\}$, ${\Bbb
D}=B(0,1)$, $S(z_0,r)=\{z\in{\Bbb C}:|z_0-z|=r\}$, $S(r)=S(0,r)$,
$R(z_0,r_1,r_2)=\{z\in{\Bbb C}:r_1<|z-z_0|<r_2\}$.

\section{Preliminaries}

Recall that a real-valued function $u$ in a domain $D$ in ${\Bbb
C}$ is said to be of {\bf bounded mean oscillation} in $D$, abbr.
$u\in{\rm BMO}(D)$, if $u\in L_{\rm loc}^1(D)$ and
\begin{equation}\label{lasibm_2.2_1}\Vert u\Vert_{*}:=
\sup\limits_{B}{\frac{1}{|B|}}\int\limits_{B}|u(z)-u_{B}|\,dm(z)<\infty\,,\end{equation}
where the supremum is taken over all discs $B$ in $D$, $dm(z)$
corresponds to the Lebesgue measure in ${\Bbb C}$ and
$$u_{B}={\frac{1}{|B|}}\int\limits_{B}u(z)\,dm(z)\,.$$ We write $u\in{\rm BMO}_{\rm loc}(D)$ if
$u\in{\rm BMO}(U)$ for every relatively compact subdomain $U$ of $D$
(we also write BMO or ${\rm BMO}_{\rm loc }$ if it is clear from the
context what $D$ is).

The class BMO was introduced by John and Nirenberg (1961) in the
paper \cite{JN} and soon became an important concept in harmonic
analysis, partial differential equations and related areas; see,
e.g., \cite{HKM} and \cite{RR}.
\medskip

A function $\varphi$ in BMO is said to have {\bf vanishing mean
oscillation}, abbr. $\varphi\in{\rm VMO}$, if the supremum in
(\ref{lasibm_2.2_1}) taken over all balls $B$ in $D$ with
$|B|<\varepsilon$ converges to $0$ as $\varepsilon\to0$. VMO has
been introduced by Sarason in \cite{Sarason}. There are a number of
papers devoted to the study of partial differential equations with
coefficients of the class VMO, see, e.g., \cite{CFL, ISbord,
MRV$^*$, Pal} and \cite{Ra}.

\medskip

\begin{remark}\label{rem1} Note that
$W^{\,1,2}\left({{D}}\right) \subset VMO \left({{D}}\right),$ see, e.g., \cite{BN}.
\end{remark}

\medskip

Following \cite{IR}, we say that a function $\varphi:D\to{\Bbb R}$
has {\bf finite mean oscillation} at a point $z_0\in D$, abbr.
$\varphi\in{\rm FMO}(z_0)$, if
\begin{equation}\label{FMO_eq2.4}\overline{\lim\limits_{\varepsilon\to0}}\ \ \
\dashint_{B(z_0,\varepsilon)}|{\varphi}(z)-\widetilde{\varphi}_{\varepsilon}(z_0)|\,dm(z)<\infty\,,\end{equation}
where \begin{equation}\label{FMO_eq2.5}
\widetilde{\varphi}_{\varepsilon}(z_0)=\dashint_{B(z_0,\varepsilon)}
{\varphi}(z)\,dm(z)\end{equation} is the mean value of the function
${\varphi}(z)$ over the disk $B(z_0,\varepsilon)$. Note that the
condition (\ref{FMO_eq2.4}) includes the assumption that $\varphi$
is integrable in some neighborhood of the point $z_0$. We say also
that a function $\varphi:D\to{\Bbb R}$ is of {\bf finite mean
oscillation in $D$}, abbr. $\varphi\in{\rm FMO}(D)$ or simply
$\varphi\in{\rm FMO}$, if $\varphi\in{\rm FMO}(z_0)$ for all points
$z_0\in D$. We write $\varphi\in{\rm FMO}(\overline{D})$ if
$\varphi$ is given in a domain $G$ in $\Bbb{C}$ such that
$\overline{D}\subset G$ and $\varphi\in{\rm FMO}(G)$.

\medskip

The following statement is obvious by the triangle inequality.

\medskip

\begin{propo}\label{FMO_pr2.1} If, for a  collection of numbers
$\varphi_{\varepsilon}\in{\Bbb R}$,
$\varepsilon\in(0,\varepsilon_0]$,
\begin{equation}\label{FMO_eq2.7}\overline{\lim\limits_{\varepsilon\to0}}\ \ \
\dashint_{B(z_0,\varepsilon)}|\varphi(z)-\varphi_{\varepsilon}|\,dm(z)<\infty\,,\end{equation} then $\varphi $
is of finite mean oscillation at $z_0$.
\end{propo}

\medskip

In particular choosing here  $\varphi_{\varepsilon}\equiv0$,
$\varepsilon\in(0,\varepsilon_0]$, we obtain the following.

\medskip

\begin{corollary}\label{FMO_cor2.1} If, for a point $z_0\in D$,
\begin{equation}\label{FMO_eq2.8}\overline{\lim\limits_{\varepsilon\to 0}}\ \ \
\dashint_{B(z_0,\varepsilon)}|\varphi(z)|\,dm(z)<\infty\,,
\end{equation} then $\varphi$ has finite mean oscillation at
$z_0$. \end{corollary}

\medskip

Recall that a point $z_0\in D$ is called a {\bf Lebesgue point} of
a function $\varphi:D\to{\Bbb R}$ if $\varphi$ is integrable in a
neighborhood of $z_0$ and \begin{equation}\label{FMO_eq2.7a}
\lim\limits_{\varepsilon\to 0}\ \ \ \dashint_{B(z_0,\varepsilon)}
|\varphi(z)-\varphi(z_0)|\,dm(z)=0\,.\end{equation} It is known
that, almost every point in $D$ is a Lebesgue point for every
function $\varphi\in L^1(D)$. Thus we have by Proposition
\ref{FMO_pr2.1} the following corollary.

\medskip

\begin{corollary}\label{FMO_cor2.7b} Every
locally integrable function $\varphi:D\to{\Bbb R}$ has a finite mean oscillation at almost every point in $D$.
\end{corollary}

\medskip

\begin{remark}\label{FMO_rmk2.13a} Note that the function $\varphi(z)=\log\left(1/|z|\right)$
belongs to BMO in the unit disk $\Delta$, see, e.g., \cite{RR}, p.
5, and hence also to FMO. However,
$\widetilde{\varphi}_{\varepsilon}(0)\to\infty$ as
$\varepsilon\to0$, showing that condition (\ref{FMO_eq2.8}) is only
sufficient but not necessary for a function $\varphi$ to be of
finite mean oscillation at $z_0$. Clearly, ${\rm BMO}(D)\subset{\rm
BMO}_{\rm loc}(D)\subset{\rm FMO}(D)$ and as well-known ${\rm
BMO}_{\rm loc}\subset L_{\rm loc}^p$ for all $p\in[1,\infty)$, see,
e.g., \cite{JN} or \cite{RR}. However, FMO is not a subclass of
$L_{\rm loc}^p$ for any $p>1$ but only of $L_{\rm loc}^1$. Thus, the
class FMO is much more wide than ${\rm BMO}_{\rm loc}$.\end{remark}

\medskip

Versions of the next lemma has been first proved for the class BMO
in \cite{RSY$_1$}. For the FMO case, see the papers \cite{IR, RS,
RSY$_7$, RSY$_4$}  and the monographs \cite{GRSY} and \cite{MRSY}.

\medskip

\begin{lemma}\label{lem13.4.2} Let $D$ be a domain in ${\Bbb C}$ and let
$\varphi:D\to{\Bbb R}$ be a  non-negative function  of the class
${\rm FMO}(z_0)$ for some $z_0\in D$. Then
\begin{equation}\label{eq13.4.5}\int\limits_{\varepsilon<|z-z_0|<\varepsilon_0}\frac{\varphi(z)\,dm(z)}
{\left(|z-z_0|\log\frac{1}{|z-z_0|}\right)^2}=O\left(\log\log\frac{1}{\varepsilon}\right)\
\quad\text{as}\quad\varepsilon\to 0\end{equation} for some $\varepsilon_0\in(0,\delta_0)$ where
$\delta_0=\min(e^{-e},d_0)$, $d_0=\sup\limits_{z\in D}|z-z_0|$. \end{lemma}

\medskip

Recall connections between some integral conditions, see, e.g.,
\cite{RSY}.

\medskip

\begin{theorem}\label{th5.555} Let $Q:{\Bbb D}\to[0,\infty]$ be a
measurable function such that \begin{equation}\label{eq5.555} \int\limits_{\Bbb
D}\Phi(Q(z))\,dm(z)<\infty\end{equation} where $\Phi:[0,\infty]\to[0,\infty]$ is a non-decreasing convex
function such that \begin{equation}\label{eq3.333a}
\int\limits_{\delta}^{\infty}\frac{d\tau}{\tau\Phi^{-1}(\tau)}=\infty\end{equation} for some $\delta>\Phi(0)$.
Then \begin{equation}\label{eq3.333A} \int\limits_0^1\frac{dr}{rq(r)}=\infty\end{equation} where $q(r)$ is the
average of the function $Q(z)$ over the circle $|z|=r$.
\end{theorem}

\medskip

\begin{remark}\label{remeq333F} Setting $H(t)=\log\Phi(t)$, note that by Theorem 2.1 in \cite{RSY}
the condition (\ref{eq3.333a}) is equivalent to each of the
conditions
\begin{equation}\label{eq333Frer}\int\limits_{\Delta}^{\infty}H'(t)\,\frac{dt}{t}=\infty,
\end{equation} \begin{equation}\label{eq333F}\int\limits_{\Delta}^{\infty}
\frac{dH(t)}{t}=\infty\,,\end{equation} 
\begin{equation}\label{eq333B}
\int\limits_{\Delta}^{\infty}H(t)\,\frac{dt}{t^2}=\infty\,\end{equation} for some $\Delta>0$, 
\begin{equation}\label{eq333C} \int\limits_{0}^{\delta}H\left(\frac{1}{t}\right)\,{dt}=\infty\end{equation} for
some $\delta>0$, \begin{equation}\label{eq333D}
\int\limits_{\Delta_*}^{\infty}\frac{d\eta}{H^{-1}(\eta)}=\infty\end{equation} for some $\Delta_*>H(+0)$. Here,
the integral in (\ref{eq333F}) is understood as the Lebesgue--Stieltjes integral and the integrals in
(\ref{eq3.333a}) and (\ref{eq333B})--(\ref{eq333D}) as the ordinary Lebesgue integrals. \end{remark}

\medskip

The following lemma is also useful, see Lemma 2.1 in \cite{KR$_2$}
or Lemma 9.2 in \cite{MRSY}.

\medskip

\begin{lemma}\label{lem8.4.1} Let $(X,\mu)$ be a measure space with a
finite measure $\mu$, $p\in(1,\infty)$ and let
$\varphi:X\to(0,\infty)$ be a measurable function. Set
\begin{equation}\label{eq8.4.2}I(\varphi,p)=\inf\limits_{\alpha}\int\limits_{X}\varphi\,\alpha^p\,d\mu\end{equation}
where the infimum is taken over all measurable functions
$\alpha:X\to[0,\infty]$ such that \begin{equation}\label{eq8.4.3}
\int\limits_{X}\alpha\,d\mu=1\,.\end{equation} Then
\begin{equation}\label{eq8.4.4}I(\varphi,p)=\left[\int\limits_{X}\varphi^{-\lambda}\,d\mu\right]^{-\frac{1}{\lambda}}\end{equation}
where \begin{equation}\label{eq8.4.5}\lambda=\frac{q}{p}\,,\qquad
\frac{1}{p}+\frac{1}{q}=1\,,\end{equation} i.e.
$\lambda=1/(p-1)\in(0,\infty)$. Moreover, the infimum in
(\ref{eq8.4.2}) is attained only for the function
\begin{equation}\label{eq8.4.6}\alpha_0=C\cdot\varphi^{-\lambda}\end{equation} where
\begin{equation}\label{eq8.4.7}C=\left(\int\limits_{X}\varphi^{-\lambda}\,d\mu\right)^{-1}\,.\end{equation}
\end{lemma}

\medskip

Finally, recall that the {\bf (conformal) modulus} of a family
$\Gamma$ of paths $\gamma$ in ${\Bbb C}$ is the quantity
\begin{equation}\label{eqModul} M(\Gamma)=\inf_{\varrho\in{\rm adm}\,\Gamma}\int\limits_{\Bbb
C}\varrho^2(z)\,dm(z)\end{equation} where a Borel function
$\varrho:{\Bbb C}\to[0,\infty]$ is {\bf admissible} for $\Gamma$,
write $\varrho\in{\rm adm}\,\Gamma$, if
\begin{equation}\label{eqAdm}\int\limits_{\gamma}\varrho\,ds\geqslant1\quad\forall\ \gamma\in\Gamma\,.\end{equation}
Here $s$ is a natural parameter of the are length on $\gamma$.

Later on, given sets $A,$ $B$ and $C$ in $\mathbb{C},$ $\Delta (A, B; C)$ denotes a family of all paths $\gamma
: [a, b] \to \mathbb{C}$ joining $A $ and $B$ in $C,$ i.e. $\gamma (a) \in A,$ $\gamma (b) \in B$ and $\gamma
(t) \in C$ for all $t \in (a,b).$

\section{On regular domains}

First of all, recall the following topological notion. A domain
$D\subset{\Bbb C}$ is said to be {\bf locally connected at a
point} $z_0\in\partial D$ if, for every neighborhood $U$ of the
point $z_0$, there is a neighborhood $V\subseteq U$ of $z_0$ such
that $V\cap D$ is connected. If this condition holds for all
$z_0\in \partial D$, then $D$ is said to be locally connected on
$\partial D$. Note that every Jordan domain $D$ in ${\Bbb C}$ is
locally connected on  $\partial D$, see, e.g., \cite{Wi}, p. 66.

We say that $\partial D$ is {\bf weakly flat at a point}
$z_0\in\partial D$ if, for every neighborhood $U$ of the point
$z_0$ and every number $P>0$, there is a neighborhood $V\subset U$
of $z_0$ such that \begin{equation}\label{eq1.5KR}
M(\Delta(E,F;D))\geqslant P\end{equation} for all continua $E$ and
$F$ in $D$ intersecting $\partial U$ and $\partial V$. We say that
$\partial D$ is {\bf weakly flat} if it is weakly flat at each
point $z_0\in\partial D$.


We also say that a point $z_0\in\partial D$ is {\bf strongly
accessible} if, for every neigh\-bor\-hood $U$ of the point $z_0$,
there exist a compactum $E$ in $D$, a neigh\-bor\-hood $V\subset U$
of $z_0$ and a number $\delta>0$ such that
\begin{equation}\label{eq1.6KR}M(\Delta(E,F;D))\geqslant\delta\end{equation}
for all continua $F$ in $D$ intersecting $\partial U$ and
$\partial V$. We say that $\partial D$ is {\bf strongly
accessible} if each point $z_0\in\partial D$ is strongly
accessible.

Here, in the definitions of strongly accessible and weakly flat
boundaries, one can take as neighborhoods $U$ and $V$ of a point
$z_0$ only balls (closed or open) centered at $z_0$ or only
neighborhoods of $z_0$ in another fundamental system of
neighborhoods of $z_0$. These conceptions can also be extended in a
natural way to the case of $\overline{\Bbb C}$ and $z_0=\infty$.
Then we must use the corresponding neighborhoods of $\infty$.

It is easy to see that if a domain $D$ in ${\Bbb C}$ is weakly flat
at a point $z_0\in\partial D$, then the point $z_0$ is strongly
accessible from $D$. Moreover, it was proved by us that if a domain
$D$ in ${\Bbb C}$ is weakly flat at a point $z_0\in\partial D$, then
$D$ is locally connected at $z_0$, see, e.g., Lemma 5.1 in
\cite{KR$_2$} or Lemma 3.15 in \cite{MRSY}.

The notions of strong accessibility and weak flatness at boundary
points of a domain in ${\Bbb C}$ defined in \cite{KR$_0$}, see also
\cite{KR$_2$} and \cite{RS}, are localizations and
ge\-ne\-ra\-li\-za\-tions of the corresponding notions introduced in
\cite{MRSY$_5$} and \cite{MRSY$_6$}, cf. with the properties $P_1$
and $P_2$ by V\"ais\"al\"a in \cite{Va} and also with the
quasiconformal accessibility and the quasiconformal flatness by
N\"akki in \cite{Na$_1$}. Many theorems on a homeomorphic extension
to the boundary of quasiconformal mappings and their generalizations
are valid under the condition of weak flatness of boundaries. The
condition of strong accessibility plays a similar role for a
continuous extension of the mappings to the boundary.

A domain $D\subset{\Bbb C}$ is called a {\bf quasiextremal distance
domain}, abbr. {\bf QED-domain}, see \cite{GM}, if
\begin{equation}\label{e:7.1}M(\Delta(E,F;{\Bbb C})\leqslant K\cdot
M(\Delta(E,F;D))\end{equation} for some $K\geqslant1$ and all pairs
of nonintersecting continua $E$ and $F$ in $D$.

It is well known, see, e.g., Theorem 10.12 in \cite{Va}, that
\begin{equation}\label{eqKPR2.2}M(\Delta(E,F;{\Bbb C}))\geqslant\frac{2}{\pi}\log{\frac{R}{r}}\end{equation}
for any sets $E$ and $F$ in ${\Bbb C}$ intersecting all the
circles $S(z_0,\rho)$, $\rho\in(r,R)$. Hence a QED-domain has a
weakly flat boundary. One example in \cite{MRSY}, Section 3.8,
shows that the inverse conclusion is not true even in the case of
simply connected domains in ${\Bbb C}$.

A domain $D\subset{\Bbb C}$ is called a {\bf uniform domain} if
each pair of points $z_1$ and $z_2\in D$ can be joined with a
rectifiable curve $\gamma$ in $D$ such that
\begin{equation}\label{e:7.2} s(\gamma)\leqslant
a\cdot|z_1-z_2|\end{equation} and \begin{equation}\label{e:7.3}
\min\limits_{i=1,2}\ s(\gamma(z_i,z))\leqslant b\cdot{\rm
dist}(z,\partial D)\end{equation} for all $z\in\gamma$ where
$\gamma(z_i,z)$ is the portion of $\gamma$ bounded by $z_i$ and
$z$, see \cite{MaSa}. It is known that every uniform domain is a
QED-domain but there exist QED-domains that are not uniform, see
\cite{GM}. Bounded convex domains and bounded domains with smooth
boundaries are simple examples of uniform domains and,
consequently, QED-domains as well as domains with weakly flat
boundaries.

It is also often met with the so-called Lipschitz domains in the
mapping theory and in the theory of differential equations. Recall
first that $\varphi:U\to{\Bbb C}$ is said to be a {\bf Lipcshitz
map} provided $|\varphi(z_1)-\varphi(z_2)|\leqslant M\cdot|z_1-z_2|$
for some $M<\infty$ and for all $z_1$ and $z_2\in U$, and a {\bf
bi-Lipcshitz map} if in addition
$M^*|z_1-z_2|\leqslant|\varphi(z_1)-\varphi(z_2)|$ for some $M^*>0$
and for all $z_1$ and $z_2\in U$. They say that $D$ in $\Bbb{C}$ is
a {\bf Lipschitz domain} if every point $z_0\in\partial D$ has a
neighborhood $U$ that can be mapped by a bi-Lipschitz homeomorphism
$\varphi$ onto the unit disk ${\Bbb D}$ in ${\Bbb C}$ such that
$\varphi(\partial D\cap U)$ is the intersection of ${\Bbb D}$ with
the real axis. Note that a bi-Lipschitz homeomorphism is
quasiconformal and, consequently, the modulus is quasiinvariant
under such a mapping. Hence the Lipschitz domains have weakly flat
boundaries.

\section{On estimates of modulus of dashed lines}

A continuous mapping $\gamma$ of an open subset $\Delta$ of the real
axis ${\Bbb R}$ or a circle into $D$ is called a {\bf dashed line},
see, e.g., Section 6.3 in \cite{MRSY}. Note that such a set $\Delta$
consists of a countable collection of mutually disjoint intervals in
${\Bbb R}$. This is the motivation for the term. The notion of the
modulus of a family $\Gamma$ of dashed lines $\gamma$ is defined
similarly to (\ref{eqModul}). We say that a property $P$ holds for
{\bf a.e.} (almost every) $\gamma\in\Gamma$ if a subfamily of all
lines in $\Gamma$ for which $P$ fails has the modulus zero, cf.
\cite{Fu}. Later on, we also say that a Lebesgue measurable function
$\varrho:{\Bbb C}\to[0,\infty]$ is {\bf extensively admissible} for
$\Gamma$, write $\varrho\in{\rm ext\,adm}\,\Gamma$, if (\ref{eqAdm})
holds for a.e. $\gamma\in\Gamma$, see, e.g., Section 9.2 in
\cite{MRSY}.

\medskip

\begin{theorem}\label{thDIR1} Let $f$ be a homeomorphic $W^{1,1}_{\rm
loc}$ solution of the Beltrami equation (\ref{eqBeltrami}) in a
domain $D\subseteq{\Bbb C}$. Then
\begin{equation}\label{LI} M(f\Sigma_{\varepsilon})\geqslant\inf\limits_{\varrho\in{\rm ext\,adm}\,\Sigma_{\varepsilon}}
\int\limits_{D}\frac{\varrho^2(z)}{K^T_{\mu}(z,z_0)}\,dm(z)\end{equation}
for all $z_0\in\overline{D}$, where
$\varepsilon\in(0,\varepsilon_0)$, $\varepsilon_0\in(0,d_0)$,
$d_0=\sup_{z\in D}|z-z_0|,$ and $\Sigma_{\varepsilon}$ denotes the
family of dashed lines consisting of all intersections of the
circles $S(z_0,r)$, $r\in(\varepsilon,\varepsilon_0)$, with $D$.
\end{theorem}

\medskip

\begin{proof} Fix  $z_0\in\overline{D}$. Let $B$ be a (Borel) set of
all points $z$ in $D$ where $f$ has a total differential with
$J_f(z)\neq0$. It is known that $B$ is the union of a countable
collection of Borel sets $B_l$, $l=1,2,\ldots$, such that
$f_l=f|_{B_l}$ is a bi-Lipschitz homeomorphism, see, e.g., Lemma
3.2.2 in \cite{Fe}. With no loss of generality, we may assume that
the $B_l$ are mutually disjoint. Denote also by $B_*$ the set of all
points $z\in D$ where $f$ has a total differential with $f'(z)=0$.

Note that the set $B_0=D\setminus(B\cup B_*)$ has the Lebesgue
measure zero in ${\Bbb C}$ by Gehring--Lehto--Menchoff theorem, see
\cite{GL} and \cite{Me}. Hence by Theorem 2.11 in \cite{KR$_2$}, see
also Lemma 9.1 in \cite{MRSY}, ${\rm length}(\gamma\cap B_0)=0$ for
a.e. paths $\gamma$ in $D$. Let us show that ${\rm
length}(f(\gamma)\cap f(B_0))=0$ for a.e. circle $\gamma$ centered
at $z_0$.

The latter follows from absolute continuity of $f$ on closed subarcs
of $\gamma\cap D$ for a.e. such a circle $\gamma$. Indeed, the class
$W^{1,1}_{\rm loc}$ is invariant with respect to local
quasi-isometries, see, e.g., Theorem 1.1.7 in \cite{Maz}, and the
functions in $W^{1,1}_{\rm loc}$ is absolutely continuous on lines,
see, e.g., Theorem 1.1.3 in \cite{Maz}. Applying say the
transformation of coordinates $\log(z-z_0)$, we come to the absolute
continuity on a.e. such circle $\gamma$. Fix $\gamma_0$ on which $f$
is absolutely continuous and ${\rm length}(\gamma_0\cap B_0)=0$.
Then ${\rm length}(f(\gamma)\cap f(B_0))={\rm length}f(\gamma_0\cap
B_0)$ and for every $\varepsilon>0$ there is an open set
$\omega_{\varepsilon}$ in $\gamma_0\cap D$ such that $\gamma_0\cap
B_0\subset\omega_{\varepsilon}$ with ${\rm
length}\,\omega_{\varepsilon}<\varepsilon$, see, e.g., Theorem
III(6.6) in \cite{Sa}. The open set $\omega_{\varepsilon}$ consists
of a countable collection of open arcs $\gamma_i$ of the circle
$\gamma_0$. By the construction $\sum\limits_i{\rm
length}\,\gamma_i<\varepsilon$ and by the absolute continuity of $f$
on $\gamma_0$ the sum $\delta=\sum\limits_i{\rm
length}\,f(\gamma_i)$ is arbitrarily small for small enough
$\varepsilon>0$. Hence ${\rm length}f(\gamma_0\cap B_0)=0$.

Thus, ${\rm length}(\gamma_*\cap f(B_0))=0$ where
$\gamma_*=f(\gamma)$ for a.e. circle $\gamma$ centered at $z_0$.
Now, let $\varrho_*\in{\rm adm}\,f(\Gamma)$ where $\Gamma$ is the
collection of all dashed lines $\gamma\cap D$ for such circles
$\gamma$ and $\varrho_*\equiv0$ outside $f(D)$. Set $\varrho\equiv0$
outside $D$ and on $B_0\cup B_{*}$ and
$$\varrho(z)\colon=\varrho_*(f(z))|\partial^{z_0}_Tf(z)|\qquad {\rm for}\ z\in B\,.$$

Arguing piecewise on $B_l$, we have by Theorem 3.2.5 under $m=1$ in
\cite{Fe} that
$$\int\limits_{\gamma}\varrho\,ds=\int\limits_{\gamma_*}\varrho_*\,ds_*\geqslant1
\qquad {\rm for\ a.e.}\ \gamma\in\Gamma$$ because ${\rm
length}(f(\gamma)\cap f(B_0))=0$ and ${\rm length}(f(\gamma)\cap
f(B_{*}))=0$ for a.e. $\gamma\in\Gamma$. Thus, $\varrho\in{\rm
ext\,adm}\,\Gamma$.

On the other hand, again arguing piecewise on $B_l$, we have  by
(\ref{Ring_eq1.5P}) that
$$\int\limits_{D}\frac{\varrho^2(z)}{K^T_{\mu}(z,z_0)}\,dm(z)\leqslant\int\limits_{f(D)}\varrho^2_*(w)\,dm(w)\,,$$
see also Lemma III.3.3 in \cite{LV}, because $\varrho(z)=0$ on
$B_0\cup B_*$. Thus, we obtain (\ref{LI}).
\end{proof}

\medskip

\begin{theorem}\label{th8.4.8} Let $f$ be a homeomorphic $W^{1,1}_{\rm
loc}$ solution of the Beltrami equation (\ref{eqBeltrami}) in a
domain $D\subseteq{\Bbb C}$.Then
\begin{equation}\label{eq8.4.9}
M(f\Sigma_{\varepsilon})\geqslant\int\limits_{\varepsilon}^{\varepsilon_0}
\frac{dr}{||K^T_{\mu}||_{1}(z_0,r)}\,,\quad\forall\
z_0\in\overline{D}\,,\ \varepsilon\in(0,\varepsilon_0)\,,\
\varepsilon_0\in(0,d_0),\end{equation} where $d_0=\sup_{z\in
D}|z-z_0|$, $\Sigma_{\varepsilon}$ denotes the family of dashed
lines consisting  of all the intersections of the circles
$S(z_0,r)$, $r\in(\varepsilon,\varepsilon_0)$, with $D$ and
\begin{equation}\label{eq8.4.11}
||K^T_{\mu}||_{1}(z_0,r):=\int\limits_{D(z_0,r)}K^T_{\mu}(z,z_0) \, |dz|\end{equation} is the norm in $L_1$ of
$K^T_{\mu}(z,z_0)$ over $D(z_0,r)=\{z\in D:|z-z_0|=r\}=D\cap S(z_0,r)$.
\end{theorem}

\medskip

\begin{proof} Indeed, for every $\varrho\in{\rm ext\,adm}\,\Sigma_{\varepsilon}$,
$$A_{\varrho}(r)=\int\limits_{D(z_0,r)}\varrho(z)\,|dz|\neq0\quad{\rm a.e.\ in}\ \
r\in(\varepsilon,\varepsilon_0)$$ is a measurable function in the
parameter $r$, say by the Fubini theorem. Thus, we may request the
equality $A_{\varrho}(r)\equiv1$ a.e. in
$r\in(\varepsilon,\varepsilon_0)$ instead of (\ref{eqAdm}) and,
thus,
$$\inf\limits_{\varrho\in{\rm ext\,adm}\,\Sigma_{\varepsilon}}\int\limits_{D\cap
R_{\varepsilon}}\frac{\varrho^2(z)}{K^T_{\mu}(z,z_0)}\,dm(z)=
\int\limits_{\varepsilon}^{\varepsilon_0}\left(\inf\limits_{\alpha\in
I(r)}\int\limits_{D(z_0,r)}\frac{\alpha^2(z)}{K^T_{\mu}(z,z_0)}\,|dz|\right)dr$$
where $R_{\varepsilon}=R(z_0,\varepsilon,\varepsilon_0)$ and $I(r)$
denotes the set of all measurable functions $\alpha$ on the dashed
line $D(z_0,r)=S(z_0,r)\cap D$ such that
$$\int\limits_{D(z_0,r)}\alpha(z)\,|dz|=1\,.$$
Hence Theorem \ref{th8.4.8} follows by Lemma \ref{lem8.4.1} with
$X=D(z_0,r)$, the length  as a measure $\mu$ on $D(z_0,r)$,
$\varphi=\frac{1}{K^T_{\mu}}|_{D(z_0,r)}$ and $p=2$.
\end{proof}

\medskip

The following lemma will be useful, too. Here we use the standard
conventions $a/\infty=0$ for $a\neq \infty$ and $a/0=\infty$ if
$a>0$ and $a\cdot\infty=0$, see, e.g., \cite{Sa}, p. 6.

\medskip

\begin{lemma}\label{lem4cr} Under the notations of Theorem
\ref{th8.4.8}, if $\| K^T_{\mu}\|_1 (z_0, r) \neq \infty$ for a.e.
$r \in (\varepsilon, \varepsilon_0),$ then
\begin{equation}\label{ring}I^{-1}=\int\limits_{A\cap
D}K^T_{\mu}(z,z_0)\cdot\eta_0^2\left(|z-z_0|\right)\,dm(z)\leqslant\int\limits_{A\cap
D}K^T_{\mu}(z,z_0)\cdot\eta^2\left(|z-z_0|\right)\,dm(z)\end{equation}
for every measurable function $\eta:(\varepsilon,
\varepsilon_0)\to[0,\infty]$ such that
\begin{equation}\label{admr}\int\limits_{\varepsilon}^{\varepsilon_0}\eta(r)\,dr=1\,,\end{equation}
where $A=R(z_0,\varepsilon, \varepsilon_0)$  and
\begin{equation}\label{eta0}
\eta_0(r)=\frac{1}{I \|K^T_{\mu}\|_1(z_0,r)}\ ,\ \ \ \ \ \ \
I=\int\limits_{\varepsilon}^{\varepsilon_0}\frac{dr}{||K^T_{\mu}||_1(z_0,r)}\,.\end{equation}
\end{lemma}

\medskip

\begin{proof} If $I=\infty$, then the left hand side in (\ref{ring}) is equal to
zero and this inequality is obvious. Hence we may assume further that $I<\infty$. Note also that
$||K^T_{\mu}||_1(z_0,r) \neq \infty$ for a.e. $r\in(\varepsilon,\varepsilon_0)$ and, consequently, $I \neq 0.$
By (\ref{admr}), $\eta(r)\ne\infty$ a.e. in $(\varepsilon,\varepsilon_0)$. We have that $\eta(r)=\alpha(r)w(r)$
a.e. in $(\varepsilon,\varepsilon_0)$ where
$$\alpha(r)=||K_{\mu}^T||_1(z_0,r)\,\eta(r)\,,\quad
w(r)=\frac{1}{||K^T_{\mu}||_1(z_0,r)}\,.$$ By the Fubini theorem in
the polar coordinates $$C\colon=\int\limits_{A\cap D}
K^T_{\mu}(z,z_0)\cdot\eta^2(|z-z_0|)\,dm(z)=\int\limits_{\varepsilon}^{\varepsilon_0}\alpha^2(r)\cdot
w(r)\,dr\,.$$

By Jensen's inequality with the weight $w(r)$, see, e.g., Theorem 2.6.2 in \cite{Ran} applied to the convex
function $\varphi(t)=t^2$ in the interval $\Omega=(r_1,r_2)$ and to the probability measure
$$\nu(E)=\frac{1}{I}\int\limits_E w(r)\,dr\,,\quad E\subset\Omega\,,$$
we obtain that $$\left(\dashint\alpha^2(r)w(r)\,dr\right)^{1/2} \geqslant\dashint\alpha(r)w(r)\,dr=\frac{1}{I}$$
where we have also used the fact that $\eta(r)=\alpha(r)\,w(r)$ satisfies (\ref{admr}). Thus, $C\geqslant
I^{-1}$ and the proof is complete. \end{proof}

\section{On a continuous extension of solutions to the boundary}

\begin{theorem}\label{thKPR9.1} Let $D$ and $D'$ be  domains in
${\Bbb C}$, $D$ be bounded and locally connected on $\partial D$ and
$\partial D'$ be strongly accessible. Suppose that $f:D\to D'$ is a
homeomorphic $W^{1,1}_{\rm loc}$ solution of the Beltrami equation
(\ref{eqBeltrami}) such that
\begin{equation}\label{eq8.11.2a}\int\limits_{0}^{\delta(z_0)}
\frac{dr}{||K^T_{\mu}||_{1}(z_0,r)}=\infty \qquad \forall \, z_0 \in
\partial D
\end{equation} for
 some $\delta(z_0)\in(0,d(z_0))$ where
$d(z_0)=\sup_{z\in D}|z-z_0|$ and
\begin{equation}\label{eq8.11.4}
||K^T_{\mu}||_1(z_0,r)=\int\limits_{D\cap
S(z_0,r)}K^T_{\mu}(z,z_0)\,|dz|\,.\end{equation} Then $f$ can be
extended to $\overline D$ by continuity in $\overline{\Bbb C}$.
\end{theorem}

\medskip

We assume that the function $K^T_{\mu}(z, z_0)$ is extended by zero
outside of $D$ in the following consequence of Theorem
\ref{thKPR9.1}.

\medskip

\begin{corollary}\label{corDIR2tcd} Let $D$ and $D'$ be  domains in
${\Bbb C}$, $D$ be bounded and locally connected on $\partial D$ and
$\partial D'$ be strongly accessible. Suppose that $f:D\to D'$ is a
homeomorphic $W^{1,1}_{\rm loc}$ solution of the Beltrami equation
(\ref{eqBeltrami}) such that
\begin{equation}\label{edgddhjgfjrqDIR6**c}
k_{z_{0}}(\varepsilon)=O\left(\left[\log\frac{1}{\varepsilon}\cdot\log\log\frac{1}{\varepsilon}\cdot\ldots\cdot\log\ldots\log\frac{1}{\varepsilon}
\right]\right) \qquad\forall\ z_0\in \partial D
\end{equation} as $\varepsilon\to0$,
where $k_{z_0}(\varepsilon)$ is the average of the function
$K^T_{\mu}(z,z_0)$ over $S(z_{0},\varepsilon)$. Then $f$ can be
extended to $\overline D$ by continuity in $\overline{\Bbb C}$.
\end{corollary}

\medskip

The proof of Theorem \ref{thKPR9.1} is reduced to the following
lemma.

\medskip

\begin{lemma}\label{lem4} Let $D$ and $D'$ be  domains in
${\Bbb C}$ and let $f:D\to D'$ be a homeomorphic $W^{1,1}_{\rm
loc}$ solution of the Beltrami equation (\ref{eqBeltrami}).
Suppose that the domain $D$ is bounded and locally connected at
$z_0\in\partial D$ and $\partial D'$ is strongly accessible at
least at one point of the cluster set
\begin{equation}\label{eq8.7.2} L:=C(z_0,f)=\{w\in\overline{\Bbb
C}:w=\lim\limits_{k\to\infty}f(z_k), z_k\to z_0\}\,.\end{equation} If the condition (\ref{eq8.11.2a}) holds for
$z_0$, then $f$ extends to $z_0$ by continuity in $\overline{\Bbb C}$.
\end{lemma}

\medskip

\begin{proof} Note that $L\neq\varnothing$ in view of compactness of
the extended plane $\overline{\Bbb C}$. By the condition $\partial
D'$ is strongly accessible at a point $\zeta_0\in L$. Let us assume
that there is one more point $\zeta_* \in L$ and set
$U=B(\zeta_*,r_0)$ where $0<r_0<|\zeta_0-\zeta_*|$.

In view of local connectedness of $D$ at $z_0$, there is a sequence of neighborhoods $V_k$ of $z_0$ with domains
$D_k=D\cap V_k$ and $ \mbox{diam} V_k\to 0$ as $k\to\infty$. Choose in the domains $D'_k=fD_k$ points $\zeta_k$
and $\zeta^*_k$ with $|\zeta_0-\zeta_k|<r_0$ and $|\zeta_0-\zeta^*_k|>r_0$, $\zeta_k\to \zeta_0$ and $\zeta^*_k
\to \zeta_*$ as $k\to\infty$. Let $C_k$ be paths connecting $\zeta_k$ and $\zeta^*_k$ in $D_k'$. Note that by
the construction $\partial U\cap C_k\neq\varnothing$. By the condition of the strong accessibility of the point
$\zeta_0$ from $D'$, there is a compactum $E\subseteq D'$ and a number $\delta>0$ such that
\begin{equation}\label{eq*}
M(\Delta(E,C_k;D'))\geq\delta \end{equation} for large $k$. Without loss of generality we may assume that the
last condition holds for all $k=1,2,\ldots$. Note that $C=f^{-1}E$ is a compactum in $D'$ and hence
$\varepsilon_0={\rm dist}(z_0,C)>0$.

Let $\Gamma_{\varepsilon}$ be the family of all paths connecting the circles $S_{\varepsilon}=\{z\in{\Bbb
C}:|z-z_0|=\varepsilon\}$ and $S_0=\{z\in\Bbb{C}:|z-z_0|=\varepsilon_0\}$ in $D$ where $\varepsilon \in (0,
\varepsilon_0)$ and $\varepsilon_0=\delta (z_0)$. Note that $C_k\subset fB_{\varepsilon}$ for every fixed
$\varepsilon\in(0,\varepsilon_0)$ for large $k$ where $B_{\varepsilon}=B(z_0,\varepsilon)$. Thus,
$M(f\Gamma_{\varepsilon})\geqslant\delta$ for all $\varepsilon\in(0,\varepsilon_0)$. However, by \cite{He} and
\cite{Zi},
\begin{equation}\label{HeZi}M(f\Gamma_{\varepsilon})\leqslant
\frac{1}{M(f\Sigma_{\varepsilon})}\end{equation} where $\Sigma_{\varepsilon}$ is the family of all dashed lines
$D(r):=\{z\in D:|z-z_0|=r\}$, $r\in(\varepsilon,\varepsilon_0)$. Thus, $M(f\Gamma_{\varepsilon})\to0$ as
$\varepsilon\to0$ by Theorem \ref{th8.4.8} in view of (\ref{eq8.11.2a}). The latter contradicts (\ref{eq*}).
This contradiction disproves the above assumption.
\end{proof}

\medskip

Combining Lemmas \ref{lem4} and \ref{lem4cr}, we come to the following general lemma where we assume that the
function $K^T_{\mu}(z,z_0)$ is extended by zero outside of the domain $D$.

\medskip

\begin{lemma}\label{lem13.3.333} Let $D$ and $D'$ be domains in
${\Bbb C}$, $D$ be locally connected on $\partial D$ and $\partial D'$ be strongly accessible. Suppose that
$f:D\to D'$ is a homeomorphic $W^{1,1}_{\rm loc}$ solution of the Beltrami equation (\ref{eqBeltrami}) such that
$\|K^T_{\mu}\|_1 (z_0,r) \neq \infty$ for a.e. $r \in (0, \varepsilon_0)$ and
\begin{equation}\label{omal}
\int\limits_{\varepsilon<|z-z_0|<\varepsilon_0}
K^T_{\mu}(z,z_0)\cdot\psi^2_{z_0,\varepsilon}(|z-z_0|)\,dm(z)=o(I_{z_0}^{2}(\varepsilon))\quad{\rm
as}\quad\varepsilon\to0\ \ \forall\ z_0\in\partial D\end{equation}
for some $\varepsilon_0\in(0,\delta_0)$ where
$\delta_0=\delta(z_0)=\sup_{z\in D}|z-z_0|$ and
$\psi_{z_0,\varepsilon}(t)$ is a family of non-negative measurable
(by Lebesgue) functions on $(0,\infty)$ such that
\begin{equation}\label{eq5.3}
I_{z_0}(\varepsilon)\colon =\int\limits_{\varepsilon}^{\varepsilon_0}
\psi_{z_0,\varepsilon}(t)\,dt<\infty\qquad\forall\ \varepsilon\in(0,\varepsilon_0)\,.\end{equation} Then $f$ can
be extended to $\overline{D}$ by continuity in $\overline{\Bbb C}$.
\end{lemma}

\medskip

Choosing in Lemma \ref{lem13.3.333} $\psi(t)=1/\left(t\, \log\left(1/t\right)\right),$ we obtain by Lemma
\ref{lem13.4.2} the following result.

\medskip

\begin{theorem}\label{thKPR9.1fmo_C} Let $D$ and $D'$ be  domains in
${\Bbb C}$, $D$ be bounded and locally connected on $\partial D$ and $\partial D'$ be strongly accessible.
Suppose that $f:D\to D'$ is a homeomorphic $W^{1,1}_{\rm loc}$ solution of the Beltrami equation
(\ref{eqBeltrami}) such that $K^T_{\mu}(z,z_0)\leqslant Q_{z_0}(z)$ a.e. in $D$ for a  function $Q_{z_0}:{\Bbb
C}\to[0,\infty]$ in the class ${\rm FMO}({z_0})$ at each point $z_0\in \partial D$. Then $f$ can be extended to
$\overline D$ by continuity in $\overline{\Bbb C}$.
\end{theorem}

\medskip

\begin{corollary}\label{corDIR1000re_c} In particular, the conclusion of
Theorem \ref{thKPR9.1fmo_C} holds if every point $z_0\in \partial D$ is the Lebesgue point of a function
$Q_{z_0}:{\Bbb C}\to[0,\infty]$ which is integrable in a neighborhood $U_{z_0}$ of the point $z_0$ such that
$K^T_{\mu}(z,z_0)\leqslant Q_{z_0}(z)$ a.e. in $D\cap U_{z_0}$.
\end{corollary}

\medskip

We assume that the function $K^T_{\mu}(z, z_0)$ is extended by zero outside of $D$ in the following consequences
of Theorem \ref{thKPR9.1fmo_C}.

\medskip

\begin{corollary}\label{corCON} Let $D$ and $D'$ be  domains in
${\Bbb C}$, $D$ be bounded and locally connected on $\partial D$ and
$\partial D'$ be strongly accessible. Suppose that $f:D\to D'$ is a
homeomorphic $W^{1,1}_{\rm loc}$ solution of the Beltrami equation
(\ref{eqBeltrami}) such that
\begin{equation}\label{eqDIR6*}\overline{\lim\limits_{\varepsilon\to0}}\quad
\dashint_{B(z_0,\varepsilon)}K^T_{\mu}(z,z_0)\,dm(z)<\infty\qquad\forall\ z_0\in \partial D\,.\end{equation}
Then $f$ can be extended to $\overline D$ by continuity in $\overline{\Bbb C}$.
\end{corollary}

\medskip

\begin{corollary}\label{corDIR2tc} Let $D$ and $D'$ be  domains in
${\Bbb C}$, $D$ be bounded and locally connected on $\partial D$ and
$\partial D'$ be strongly accessible. Suppose that $f:D\to D'$ is a
homeomorphic $W^{1,1}_{\rm loc}$ solution of the Beltrami equation
(\ref{eqBeltrami}) such that
\begin{equation}\label{eqDIR61c}k_{z_{0}}(\varepsilon)=O\left(\log\frac{1}{\varepsilon}\right)
\qquad\mbox{as}\ \varepsilon\to0\qquad\forall\ z_0\in \partial D
\end{equation}
where $k_{z_0}(\varepsilon)$ is the average of the function $K^T_{\mu}(z,z_0)$ over $S(z_{0},\varepsilon)$. Then
$f$ can be extended to $\overline D$ by continuity in $\overline{\Bbb C}$.
\end{corollary}

\medskip

\begin{remark}\label{rem2c} In particular, the conclusion of Corollary \ref{corDIR2tc} holds if
\begin{equation}\label{eqDIR6**c} K^T_{\mu}(z,z_0)=O\left(\log\frac{1}{|z-z_0|}\right)\qquad{\rm
as}\quad z\to z_0\quad\forall\ z_0\in
\partial D\,.\end{equation}\end{remark}

\medskip

Similarly, choosing in Lemma \ref{lem13.3.333} the function
$\psi(t)=1/t$, we come to the following statement.

\medskip

\begin{theorem}\label{thKPRS12b*} Let $D$ and $D'$ be domains in
${\Bbb C}$, $D$ be bounded and locally connected on $\partial D$ and
$\partial D'$ be strongly accessible. Suppose that $f:D\to D'$ is a
homeomorphic $W^{1,1}_{\rm loc}$ solution of the Beltrami equation
(\ref{eqBeltrami}) such that
\begin{equation}\label{eqKPRS12c*}
\int\limits_{\varepsilon<|z-z_0|<\varepsilon_0}K^T_{\mu}(z,z_0)\,\frac{dm(z)}{|z-z_0|^2}
=o\left(\left[\log\frac{1}{\varepsilon}\right]^2\right)\qquad\forall\ z_0\in\partial D\end{equation} as
$\varepsilon\to 0$ for some $\varepsilon_0=\delta(z_0)\in(0,d(z_0))$ where $d(z_0)=\sup_{z\in D}|z-z_0|$. Then
$f$ can be extended to $\overline D$ by continuity in $\overline{\Bbb C}$. \end{theorem}

\medskip

\begin{remark}\label{rmKRRSa*} Choosing in Lemma \ref{lem13.3.333} the function
$\psi(t)=1/(t\log{1/t})$ instead of $\psi(t)=1/t$, we are able to
replace (\ref{eqKPRS12c*}) by \begin{equation}\label{eqKPRS12f*}
\int\limits_{\varepsilon<|z-z_0|<\varepsilon_0}\frac{K^T_{\mu}(z,z_0)\,dm(z)}{\left(|z-z_0|\log{\frac{1}{|z-z_0|}}\right)^2}
=o\left(\left[\log\log\frac{1}{\varepsilon}\right]^2\right)\end{equation}
In general, we are able to give here the whole scale of the
corresponding conditions in $\log$ using functions $\psi(t)$ of the
form
$1/(t\log{1}/{t}\cdot\log\log{1}/{t}\cdot\ldots\cdot\log\ldots\log{1}/{t})$.
\end{remark}

\medskip

Finally, combining Theorems \ref{th5.555} and \ref{thKPR9.1}, we
obtain the following.

\medskip

\begin{theorem}\label{thKR4.1c} {\it Let $D$ and $D'$ be  domains in
${\Bbb C}$, $D$ be locally connected on $\partial D$ and $\partial
D'$ be strongly accessible. Suppose that $f:D\to D'$ is a
homeomorphic $W^{1,1}_{\rm loc}$ solution of the Beltrami equation
(\ref{eqBeltrami}) such that
\begin{equation}\label{eqKR4.1shc}\int\limits_{D\cap U_{z_0}}
\Phi_{z_0}\left(K^T_{\mu}(z,z_0)\right)\,dm(z)<\infty\quad\forall\
z_0\in\partial{D}\end{equation} for a convex non-decreasing function
$\Phi_{z_0}:[0,\infty]\to[0,\infty]$ and a neighborhood $U_{z_0}$ of
the point $z_0$. If
\begin{equation}\label{eqKR4.2rasc}
\int\limits_{\delta(z_0)}^{\infty}\frac{d\tau}{\tau\Phi_{z_0}^{-1}(\tau)}=\infty
\qquad\forall\ z_0\in \partial D\end{equation} for some
$\delta(z_0)>\Phi_{z_0}(0)$. Then $f$ can be extended to $\overline
D$ by continuity in $\overline{\Bbb C}$.}
\end{theorem}

\medskip

\begin{corollary}\label{corDIR1000c} In particular, the conclusion of
Theorem \ref{thKR4.1c} holds if
\begin{equation}\label{eqKR4.1c}\int\limits_{D\cap U_{z_0}}e^{\alpha(z_0) K^T_{\mu}(z,z_0)}\,dm(z)<\infty
\qquad\forall\ z_0\in \partial D\end{equation} for some $\alpha(z_0)>0$ and a neighborhood $U_{z_0}$ of the
point $z_0$.
\end{corollary}

\medskip

\section{The extension of the inverse mappings to the boundary}

\medskip

Let us start from the following fine lemma.

\medskip

\begin{lemma}\label{lemKPR8.1} Let $D$ and $D'$ be domains in
${\Bbb C}$, $z_1$ and $z_2$ be distinct points in $\partial D$,
$z_1\neq\infty$, and let $f:D\to D'$ be a homeomorphic $W^{1,1}_{\rm
loc}$ solution of the Beltrami equation (\ref{eqBeltrami}). Suppose
that the function $K^T_{\mu}(z,z_1)$ is integrable on the dashed
lines
\begin{equation}\label{eqKPR6.1}D(z_1, r)\colon =\{z\in D:|z-z_1|=r\}=D\cap S(z_1,r)\end{equation}
for some set $E$ of numbers $r<|z_1-z_2|$ of a positive linear measure. If $D$ is locally connected at $z_1$ and
$z_2$ and $\partial D'$ is weakly flat, then
\begin{equation}\label{eq8.10.2} C(z_1,f)\cap C(z_2,f)=\varnothing.\end{equation}
\end{lemma}

\medskip

\begin{proof} Without loss of generality, we may assume that the
domain $D$ is bounded. Let $d=|z_1-z_2|$. Choose
$\varepsilon_0\in(0,d)$ and $\varepsilon\in(0,\varepsilon_0)$ such
that $$E_0:=\{r\in E:r\in(\varepsilon,\varepsilon_0)\}$$ has a
positive measure. The choice is possible because of a countable
subadditivity of the linear measure and because of the exhaustion
of $E$ by the sets  $$E_m:=\{r\in E:r\in(1/m,d-1/m)\}\,.$$ Note
that each of the circles $S(z_1,r)$, $r\in E_0$, separates the
points $z_1$ and $z_2$ in ${\Bbb C}$ and $D(r)$, $r\in E_0$, in
$D$. Thus, by Theorem \ref{th8.4.8} we have that
\begin{equation}\label{eq8.10.3}M(f\Sigma_{\varepsilon})>0\end{equation} where
$\Sigma_{\varepsilon}$ denotes the family of all intersections of
$D$ with the circles $$S(z_1,r)=\{z\in{\Bbb C}:|z-z_1|=r\}\,,\quad
r\in(\varepsilon,\varepsilon_0)\,.$$

For $i=1,2$, let $C_i$ be the cluster set $C(z_i,f)$ and suppose
that $C_1\cap C_2\neq\varnothing$. Since $D$ is locally connected
at $z_1$ and $z_2$, there exist neighborhoods $U_i$ of $z_i$ such
that $W_i=D\cap U_i$, $i=1,2$ are connected and $U_1\subset
B(z_1,\varepsilon)$ and $U_2\subset{\Bbb C}\setminus
B(z_1,\varepsilon_0)$. Set
$\Gamma=\Delta(\overline{W_1},\overline{W_2};D)$. By \cite{He} and
\cite{Zi} and (\ref{eq8.10.3}) \begin{equation}\label{eq8.10.4}
M(f\Gamma)\leqslant\frac{1}{M(f\Sigma_{\varepsilon})}<\infty\,.\end{equation}
Let $\zeta_0\in C_1\cap C_2$. Without loss of generality, we may
assume that $\zeta_0\neq\infty$ because in the contrary case one
can use an additional M\"{o}bius transformation. Choose $r_0>0$
such that $S(\zeta_0,r_0)\cap fW_1\neq\varnothing$ and
$S(\zeta_0,r_0)\cap fW_2\neq\varnothing$.

By the condition $\partial D'$ is weakly flat and hence, given a
finite number $M_0>M(f\Gamma)$, there is $r_*\in(0,r_0)$ such that
$$M(\Delta(E,F;D'))\geqslant M_0$$ for all continua $E$ and $F$ in
$D'$ intersecting the circles $S(\zeta_0,r_0)$ and $S(\zeta_0,r_*)$.
However, these circles can be connected by paths $P_1$ and $P_2$ in
the domains $fW_1$ and $fW_2$, respectively, and for those paths
$$M_0\leqslant M(\Delta(P_1,P_2;D'))\leqslant M(f\Gamma)\,.$$

The contradiction disproves the above assumption that $C_1\cap
C_2\neq\varnothing$. The proof is complete.
\end{proof}

\medskip

As an immediate consequence of Lemma \ref{lemKPR8.1}, we have the
following statement.


\begin{theorem}\label{thKPR8.2} Let $D$ and $D'$ be domains in
${\Bbb C}$, $D$ locally connected on $\partial D$ and $\partial D'$ weakly flat. Suppose that $f:D\to D'$ is a
homeomorphic $W^{1,1}_{\rm loc}$ solution of the Beltrami equation (\ref{eqBeltrami}) with $K^T_{\mu}(z,z_0)\in
L^1(D\cap U_{z_0})$ for a neighborhood $U_{z_0}$ of every point $z_0\in \partial D$. Then $f^{-1}$ has an
extension to $\overline{D'}$ by continuity in $\overline{\Bbb C}$. \end{theorem}


\begin{proof} By the Fubini theorem with notations from Lemma
\ref{lemKPR8.1}, the set \begin{equation}\label{eqKPR6.2a}
E=\{r\in(0,d):K^T_{\mu}(z,z_0)|_{D(z_0, r)}\in L^{1}(D(z_0,
r))\}\end{equation} has a positive linear measure because
$K^T_{\mu}(z,z_0)\in L^1(D\cap U_{z_0})$. Consequently, arguing by
contradiction, we obtain the desired conclusion on the basis of
Lemma \ref{lemKPR8.1}.
\end{proof}


Moreover, by Lemma \ref{lemKPR8.1} we obtain also the following
conclusion.


\begin{theorem}\label{thKPR8.3} Let $D$ and $D'$ be domains in
${\Bbb C}$, $D$ bounded and locally connected on $\partial D$ and $\partial D'$ weakly flat. Suppose that
$f:D\to D'$ is a homeomorphic $W^{1,1}_{\rm loc}$ solution of the Beltrami equation (\ref{eqBeltrami}) with the
coefficient $\mu$ such that the condition (\ref{eq8.11.2a}) holds for all $z_0\in\partial D$. Then there is an
extension of $f^{-1}$ to $\overline{D'}$ by continuity in $\overline{\Bbb C}$. \end{theorem}


Combining Theorem \ref{thKPR8.3} and Lemma \ref{lem4cr}, we come to
the following general lemma where we assume as above that
$K^T_{\mu}(z,z_0)$ is extended by zero outside of the domain $D$.


\begin{lemma}\label{lem13.5.333} Let $D$ and $D'$ be domains in
${\Bbb C}$, $D$ be locally connected on $\partial D$ and $\partial D'$ be a weakly flat. Suppose that $f:D\to
D'$ is a homeomorphic $W^{1,1}_{\rm loc}$ solution of the Beltrami equation (\ref{eqBeltrami}) such that
$\|K^T_{\mu}\| (z_0,r) \neq \infty$ for a.e. $r \in (0, \varepsilon_0)$ and \begin{equation}\label{omal}
\int\limits_{\varepsilon<|z-z_0|<\varepsilon_0}
K^T_{\mu}(z,z_0)\cdot\psi^2_{z_0,\varepsilon}(|z-z_0|)\,dm(z)=o(I_{z_0}^{2}(\varepsilon))\quad{\rm
as}\quad\varepsilon\to0\ \ \forall\ z_0\in\partial D\end{equation} for some $\varepsilon_0\in(0,\delta_0)$ where
$\delta_0=\delta(z_0)=\sup_{z\in D}|z-z_0|$ and $\psi_{z_0,\varepsilon}(t)$ is a family of non-negative
measurable (by Lebesgue) functions on $(0,\infty)$ such that
\begin{equation}\label{eq5.3}
I_{z_0}(\varepsilon)\colon =\int\limits_{\varepsilon}^{\varepsilon_0}
\psi_{z_0,\varepsilon}(t)\,dt<\infty\qquad\forall\ \varepsilon\in(0,\varepsilon_0)\,.\end{equation} Then there
is an extension of $f^{-1}$ to $\overline{D'}$ by continuity in $\overline{\Bbb C}$.  \end{lemma}

\medskip

\section{On a homeomorphic extension to the boundary}

\medskip

Combining Lemmas \ref{lem13.3.333} and \ref{lem13.5.333}, we obtain
one more important general lemma where as we assume that
$K^T_{\mu}(z,z_0)$ is extended by zero outside $D$.

\medskip

\begin{lemma}\label{lem3.3.333} Let $D$ and $D'$ be domains in
${\Bbb C}$, $D$ be locally connected on $\partial D$ and $\partial D'$ be weakly flat. Suppose that $f:D\to D'$
is a homeomorphic $W^{1,1}_{\rm loc}$ solution of the Beltrami equation (\ref{eqBeltrami}) such that
$\|K^T_{\mu}\| (z_0,r) \neq \infty$ for a.e. $r \in (0, \varepsilon_0)$ and
\begin{equation}\label{3omal}
\int\limits_{\varepsilon<|z-z_0|<\varepsilon_0}
K^T_{\mu}(z,z_0)\cdot\psi^2_{z_0,\varepsilon}(|z-z_0|)\,dm(z)=o(I_{z_0}^{2}(\varepsilon))\quad{\rm
as}\quad\varepsilon\to0\ \ \forall\ z_0\in\partial D\end{equation}
for some $\varepsilon_0\in(0,\delta_0)$ where
$\delta_0=\delta(z_0)=\sup_{z\in D}|z-z_0|$ and
$\psi_{z_0,\varepsilon}(t)$ is a family of non-negative measurable
(by Lebesgue) functions on $(0,\infty)$ such that
\begin{equation}\label{eq3.5.3}
I_{z_0}(\varepsilon)\colon
=\int\limits_{\varepsilon}^{\varepsilon_0}
\psi_{z_0,\varepsilon}(t)\,dt<\infty\qquad\forall\
\varepsilon\in(0,\varepsilon_0)\,.\end{equation} Then $f$ can be
extended to a homeomorphism $\overline{f}:
\overline{D}\to\overline{D'}$ by continuity in $\overline{\Bbb C}$.
\end{lemma}

\medskip

Arguing as in Section 5 and combining the corresponding results of
Sections 5 and 6, we obtain the following results.

\medskip

\begin{theorem}\label{th7.KPR9.1fmo_C} Let $D$ and $D'$ be  domains in
${\Bbb C}$, $D$ be bounded and locally connected on $\partial D$ and $\partial D'$ be weakly flat. Suppose that
$f:D\to D'$ is a homeomorphic $W^{1,1}_{\rm loc}$ solution of the Beltrami equation (\ref{eqBeltrami}) such that
$K^T_{\mu}(z,z_0)\leqslant Q_{z_0}(z)$ a.e. in $D$ for a  function $Q_{z_0}:{\Bbb C}\to[0,\infty]$ in the class
${\rm FMO}({z_0})$ at each point $z_0\in \partial D$. Then $f$ can be extended to a homeomorphism $\overline{f}:
\overline{D}\to\overline{D'}$ by continuity in $\overline{\Bbb C}$.
\end{theorem}


\begin{corollary}\label{cor7.DIR1000re_c} In particular, the conclusion of
Theorem \ref{th7.KPR9.1fmo_C} holds if every point $z_0\in \partial D$ is the Lebesgue point of a function
$Q_{z_0}:{\Bbb C}\to[0,\infty]$ which is integrable in a neighborhood $U_{z_0}$ of the point $z_0$ such that
$K^T_{\mu}(z,z_0)\leqslant Q_{z_0}(z)$ a.e. in $D\cap U_{z_0}$.
\end{corollary}

\medskip
\begin{corollary}\label{cor7.CON} Let $D$ and $D'$ be  domains in
${\Bbb C}$, $D$ be bounded and locally connected on $\partial D$ and
$\partial D'$ be weakly flat. Suppose that $f:D\to D'$ is a
homeomorphic $W^{1,1}_{\rm loc}$ solution of the Beltrami equation
(\ref{eqBeltrami}) such that
\begin{equation}\label{eq7.DIR6*}\overline{\lim\limits_{\varepsilon\to0}}\quad
\dashint_{B(z_0,\varepsilon)}K^T_{\mu}(z,z_0)\,dm(z)<\infty\qquad\forall\ z_0\in \partial D\,.\end{equation}
Then $f$ can be extended to a homeomorphism $\overline{f}: \overline{D}\to\overline{D'}$ by continuity in
$\overline{\Bbb C}$.
\end{corollary}

\medskip

\begin{corollary}\label{cor7.DIR2tc} Let $D$ and $D'$ be  domains in
${\Bbb C}$, $D$ be bounded and locally connected on $\partial D$ and
$\partial D'$ be weakly flat. Suppose that $f:D\to D'$ is a
homeomorphic $W^{1,1}_{\rm loc}$ solution of the Beltrami equation
(\ref{eqBeltrami}) such that
\begin{equation}\label{eq7.DIR61c} k_{z_{0}}(\varepsilon)=O\left(\log\frac{1}{\varepsilon}\right)
\qquad\mbox{as}\ \varepsilon\to0\qquad\forall\ z_0\in \partial D
\end{equation}
where $k_{z_0}(\varepsilon)$ is the average of the function $K^T_{\mu}(z,z_0)$ over $S(z_{0},\varepsilon)$. Then
$f$ can be extended to a homeomorphism $\overline{f}: \overline{D}\to\overline{D'}$ by continuity in
$\overline{\Bbb C}$.
\end{corollary}

\medskip

\begin{remark}\label{rem7.2c} In particular, the conclusion of Corollary \ref{cor7.DIR2tc} holds if
\begin{equation}\label{eq7.DIR6**c} K^T_{\mu}(z,z_0)=O\left(\log\frac{1}{|z-z_0|}\right)\qquad{\rm
as}\quad z\to z_0\quad\forall\ z_0\in
\partial D\,.\end{equation}\end{remark}

\medskip

\begin{theorem}\label{th7.KPRS12b*} Let $D$ and $D'$ be domains in
${\Bbb C}$, $D$ be bounded and locally connected on $\partial D$ and
$\partial D'$ be weakly flat. Suppose that $f:D\to D'$ is a
homeomorphic $W^{1,1}_{\rm loc}$ solution of the Beltrami equation
(\ref{eqBeltrami}) such that
\begin{equation}\label{eq7.KPRS12c*}
\int\limits_{\varepsilon<|z-z_0|<\varepsilon_0}K^T_{\mu}(z,z_0)\,\frac{dm(z)}{|z-z_0|^2}
=o\left(\left[\log\frac{1}{\varepsilon}\right]^2\right)\qquad\forall\ z_0\in\partial D\end{equation} as
$\varepsilon\to 0$ for some $\varepsilon_0=\delta(z_0)\in(0,d(z_0))$ where $d(z_0)=\sup_{z\in D}|z-z_0|$. Then
$f$ can be extended to a homeomorphism $\overline{f}: \overline{D}\to\overline{D'}$ by continuity in
$\overline{\Bbb C}$. \end{theorem}

\medskip

\begin{remark}\label{rm7.KRRSa*} Choosing in Lemma \ref{lem3.3.333} the function
$\psi(t)=1/(t\log{1/t})$ instead of $\psi(t)=1/t$, we are able to
replace (\ref{eq7.KPRS12c*}) by \begin{equation}\label{eq7.KPRS12f*}
\int\limits_{\varepsilon<|z-z_0|<\varepsilon_0}\frac{K^T_{\mu}(z,z_0)\,dm(z)}{\left(|z-z_0|\log{\frac{1}{|z-z_0|}}\right)^2}
=o\left(\left[\log\log\frac{1}{\varepsilon}\right]^2\right)\end{equation}
Again, we are able also to give here the whole scale of the
corresponding conditions in $\log$ using functions $\psi(t)$ of the
form
$1/(t\log{1}/{t}\cdot\log\log{1}/{t}\cdot\ldots\cdot\log\ldots\log{1}/{t})$.
\end{remark}

\medskip

\begin{theorem}\label{th7.KR4.1c} Let $D$ and $D'$ be  domains in
${\Bbb C}$, $D$ be bounded and locally connected on $\partial D$ and
$\partial D'$ be weakly flat. Suppose that $f:D\to D'$ is a
homeomorphic $W^{1,1}_{\rm loc}$ solution of the Beltrami equation
(\ref{eqBeltrami}) such that
\begin{equation}\label{eq7.8.11.2a}\int\limits_{0}^{\delta(z_0)}
\frac{dr}{||K^T_{\mu}||_{1}(z_0,r)}=\infty \qquad\forall\ z_0\in
\partial D \end{equation} for some $\delta(z_0)\in(0,d(z_0))$ where
$d(z_0)=\sup_{z\in D}|z-z_0|$ and
\begin{equation}\label{eq7.8.11.4}
||K^T_{\mu}||_1(z_0,r)=\int\limits_{D\cap S(z_0,r)}K^T_{\mu}(z,z_0)\,|dz|\,.\end{equation} Then $f$ can be
extended to a homeomorphism $\overline{f}: \overline{D}\to\overline{D'}$ by continuity in $\overline{\Bbb C}$.
\end{theorem}

\medskip

\begin{corollary}\label{cor7.DIR2tcd} Let $D$ and $D'$ be  domains in
${\Bbb C}$, $D$ be bounded and locally connected on $\partial D$ and
$\partial D'$ be weakly flat. Suppose that $f:D\to D'$ is a
homeomorphic $W^{1,1}_{\rm loc}$ solution of the Beltrami equation
(\ref{eqBeltrami}) such that
\begin{equation}\label{7.edgddhjgfjrqDIR6**c}
k_{z_{0}}(\varepsilon)=O\left(\left[\log\frac{1}{\varepsilon}\cdot\log\log\frac{1}{\varepsilon}\cdot\ldots\cdot\log\ldots\log\frac{1}{\varepsilon}
\right]\right) \qquad\forall\ z_0\in \partial D
\end{equation} as $\varepsilon\to0$,
where $k_{z_0}(\varepsilon)$ is the average of the function
$K^T_{\mu}(z,z_0)$ over $S(z_{0},\varepsilon)$. Then $f$ can be
extended to a homeomorphism $\overline{f}:
\overline{D}\to\overline{D'}$ by continuity in $\overline{\Bbb C}$.
\end{corollary}

\medskip

As usual, we assume here that $K^T_{\mu}(z,z_0)$ is extended by zero
outside of $D$.

\medskip

\begin{theorem}\label{th7.KR4.1cr} Let $D$ and $D'$ be  domains in
${\Bbb C}$, $D$ be locally connected on $\partial D$ and $\partial
D'$ be weakly flat. Suppose that $f:D\to D'$ is a homeomorphic
$W^{1,1}_{\rm loc}$ solution of the Beltrami equation
(\ref{eqBeltrami}) such that
\begin{equation}\label{eqKR7.4.1shc}\int\limits_{D\cap U_{z_0}}
\Phi_{z_0}\left(K^T_{\mu}(z,z_0)\right)\,dm(z)<\infty \qquad\forall\
z_0\in \partial D\end{equation} for a convex non-decreasing function
$\Phi_{z_0}:[0,\infty]\to[0,\infty]$ and a neighborhood $U_{z_0}$ of
the point $z_0$. If
\begin{equation}\label{eqKR7.4.2rasc}
\int\limits_{\delta(z_0)}^{\infty}\frac{d\tau}{\tau\Phi_{z_0}^{-1}(\tau)}=\infty \qquad\forall\ z_0\in \partial
D\end{equation} for some $\delta(z_0)>\Phi_{z_0}(0)$. Then $f$ can be extended to a homeomorphism $\overline{f}:
\overline{D}\to\overline{D'}$ by continuity in $\overline{\Bbb C}$.
\end{theorem}

\medskip

\begin{corollary}\label{cor7.DIR1000c} In particular, the conclusion of
Theorem \ref{th7.KR4.1c} holds if
\begin{equation}\label{eq7.KR4.1c}\int\limits_{D\cap U_{z_0}}e^{\alpha(z_0) K^T_{\mu}(z,z_0)}\,dm(z)<\infty
\qquad\forall\ z_0\in \partial D\end{equation} for some $\alpha(z_0)>0$ and a neighborhood $U_{z_0}$ of the
point $z_0$.
\end{corollary}

\medskip

The above results are far reaching generalizations of the well-known Gehring--Martio theorem on a homeomorphic
extension to the closure of quasiconformal mappings between quasiextremal distance domains, see \cite{GM}, see
also \cite{MV}, specified to the plane.

\begin{remark}\label{rm7.KR4.1} By Theorem 5.1 and Remark 5.1 in
\cite{KR$_3$} the condition (\ref{eqKR7.4.2rasc}) are not only
sufficient but also necessary for continuous extension  to the
boundary of $f$ with the integral constraint (\ref{eqKR7.4.1shc}).
Note, by Remark \ref{remeq333F} the condition (\ref{eqKR7.4.2rasc})
is equivalent to each of the conditions
(\ref{eq333Frer})--(\ref{eq333D}) with $\Phi(z)=\Phi_{z_0}(z)$.

Finally, note that all the above results can be formulated in terms
of the dilatation $K_{\mu}(z)$ instead of $K^T_{\mu}(z,z_0)$ because
$K^T_{\mu}(z,z_0) \leqslant K_{\mu}(z)$ for all $z_0\in \Bbb C$ and
$z\in D$. These results are true, in particular, for many regular
domains $D$ and $D'$ as convex, smooth and Lipschitz, uniform and
quasiextremal distance domains by Gehring--Martio. It is important
for further applications to the Dirichlet problem for the degenerate
Beltrami equations that the first of the domains $D$ can be an
arbitrary Jordan domain or a bounded domain whose boundary consists
of a finite collection of mutually disjoint Jordan curves.
\end{remark}


\section{On regular solutions for the Dirichlet problem in the Jordan domains}

In this section, we prove that a regular solution of the Dirichlet
problem (\ref{eqGrUsl}) exists for every continuous function
$\varphi:\partial D\to{\Bbb R}$ for wide classes of the degenerate
Beltrami equations (\ref{eqBeltrami}) in an arbitrary  Jordan domain
$D$. The main criteria are formulated by us in terms of the tangent
dilatations $K^T_{\mu}(z,z_0)$ which are more refined although the
corresponding criteria remain valid, in view of (\ref{eqConnect}),
for the usual dilatation $K_{\mu}(z)$, cf. \cite{KPR1}. We assume
that the dilatations $K^T_{\mu}(z,z_0)$ and $K_{\mu}(z)$ are
extended by zero outside of the domain $D$ in the following lemma
and remark.

\medskip

\begin{lemma}\label{lemDIR9} Let $D$ be a Jordan domain in
${\Bbb C}.$ Suppose that $\mu:D\to{\Bbb C}$ is a measurable function with $|\mu(z)|<1$ a.e.,  $K_{\mu}(z)\in
L^1_{\rm loc}(D),$ $K^T_{\mu}(z,z_0)$ is integrable over $D\bigcap S (z_0, r)$ for a.e. $r \in (0,
\varepsilon_0)$ and
\begin{equation}\label{3omal}
\int\limits_{\varepsilon<|z-z_0|<\varepsilon_0}
K^T_{\mu}(z,z_0)\cdot\psi^2_{z_0,\varepsilon}(|z-z_0|)\,dm(z)=o(I_{z_0}^{2}(\varepsilon))\quad{\rm
as}\quad\varepsilon\to0\ \ \forall\ z_0\in\overline{D}\end{equation}
for some $\varepsilon_0\in(0,\delta_0)$ where
$\delta_0=\delta(z_0)=\sup_{z\in D}|z-z_0|$ and
$\psi_{z_0,\varepsilon}(t)$ is a family of non-negative measurable
(by Lebesgue) functions on $(0,\infty)$ such that
\begin{equation}\label{eq3.5.3}
I_{z_0}(\varepsilon)\colon =\int\limits_{\varepsilon}^{\varepsilon_0}
\psi_{z_0,\varepsilon}(t)\,dt<\infty\qquad\forall\ \varepsilon\in(0,\varepsilon_0)\,.\end{equation} Then the
Beltrami equation (\ref{eqBeltrami}) has a regular solution $f$ of the Dirichlet problem (\ref{eqGrUsl}) for
each continuous function $\varphi:\partial D\to{\Bbb R}$.\end{lemma}

\medskip

\begin{remark}\label{rem333} Note that if the family of the functions
$\psi_{z_0,\varepsilon}(t)\equiv\psi(t)$ is independent on the
parameters $z_0$ and $\varepsilon$, then the condition (\ref{3omal})
implies that $I_{z_0}(\varepsilon)\to \infty$ as $\varepsilon\to 0$.
This follows immediately from arguments by contradiction. Note also
that (\ref{3omal}) holds, in particular, if
\begin{equation}\label{333omal}
\int\limits_{\varepsilon<|z-z_0|<\varepsilon_0} K_{\mu}(z)\cdot\psi^2(|z-z_0|)\,dm(z)<\infty \qquad \forall\
z_0\in\overline{D}\end{equation} and $I_{z_0}(\varepsilon)\to \infty$ as $\varepsilon\to 0$. In other words, for
the solvability of the Dirichlet problem (\ref{eqGrUsl}) for the Beltrami equation (\ref{eqBeltrami}) for all
continuous $\varphi(\zeta)\not\equiv{\rm const}$, it is sufficient that the integral in (\ref{333omal})
converges in the sense of the principal value for some nonnegative function $\psi(t)$ that is locally integrable
over $(0,\varepsilon_0 ]$ but has a nonintegrable singularity at $0$. The functions $\log^{\lambda}(e/|z-z_0|)$,
$\lambda\in (0,1)$, $z\in\Bbb D$, $z_0\in\overline{\Bbb D}$, and $\psi(t)=1/(t \,\, \log(e/t))$, $t\in(0,1)$,
show that the condition (\ref{333omal}) is compatible with the condition $I_{z_0}(\varepsilon)\to\infty $ as
$\varepsilon\to 0$. Furthermore, the condition (\ref{3omal}) shows that it is sufficient for the solvability of
the Dirichlet problem even that the integral in (\ref{333omal}) was divergent in a controlled way.
\end{remark}

\medskip

\begin{proof}[Proof of Lemma \ref{lemDIR9}] Let $F$ be a regular
homeomorphic solution of the Beltrami equation (\ref{eqBeltrami}) of
the class $W^{1,1}_{\rm loc}$ that exists by Lemma 4.1 in
\cite{RSY$_6$}. Note that $\overline{\Bbb C}\setminus D^*$, where
$D^*=F(D)$, cannot consist of the single point $\infty$ because in
the contrary case $\partial D^*$ would be weakly flat. But then by
Lemma \ref{lem13.3.333} $F$ should have a homeomorphic extension to
$\overline{D}$ that is impossible because $\partial D$ is not a
singleton. Moreover, the domain $D^*$ is simply connected, see,
e.g., either Lemma 5.3 in \cite{IR} or Lemma 6.5 in \cite{MRSY}.
Thus, by the Riemann theorem, see, e.g., Theorem II.2.1 in
\cite{Gol}, $D^*$ can be mapped by a conformal mapping $R$ onto the
unit disk ${\Bbb D}$. The mapping $g=R\circ F$ is also a regular
homeomorphic solution of the Beltrami equation of the class $W_{\rm
loc}^{1,1}$ that maps $D$ onto ${\Bbb D}$. Furthermore, by Lemma 7.1
$g$ admits a homeomorphic extension
$g_*:\overline{D}\to\overline{\Bbb D}$ because ${\Bbb D}$ has a
weakly flat boundary and the Jordan domain $D$ is locally connected
on its boundary.

Let us find a solution of the Dirichlet problem (\ref{eqGrUsl}) in
the form $f=h\circ g$ where $h$ is an analytic function in ${\Bbb
D}$ with the boundary condition $$\lim\limits_{z\to\zeta}{\rm
Re}\,h(z)=\varphi(g_{*}^{-1}(\zeta)) \qquad\forall\
\zeta\in\partial{\Bbb D}\,.$$ By the Schwarz formula (see, e.g.,
Section 8, Chapter III, Part 3 in \cite{HuCo}), the analytic
function $h$ with ${\rm Im}\,h(0)=0$ can be calculated in ${\Bbb D}$
through its real part on the boundary:
\begin{equation}\label{eqDIR4*} h(z)=\frac{1}{2\pi
i}\int\limits_{|\zeta|=1}{\rm Re}\,\varphi\circ
g_{*}^{-1}(\zeta)\cdot\frac{\zeta+z}{\zeta-z}\cdot\frac{d\zeta}{\zeta}\,.\end{equation}
We see that the function $f=h\circ g$ is the desired regular
solution of the Dirichlet problem (\ref{eqGrUsl}) for the Beltrami
equation (\ref{eqBeltrami}).
\end{proof}

Choosing in Lemma \ref{lemDIR9} $\psi(t)=1/\left(t\, \log\left(1/t\right)\right)$, we obtain by Lemma \ref{lem13.4.2}
 the following result.

\medskip

\begin{theorem}\label{thDIR2fmo} Let $D$ be a Jordan domain and $\mu:D\to{\Bbb C}$
be a measurable function with $|\mu(z)|<1$ a.e. such that $K_{\mu}\in L^1_{\rm loc}(D)$. Suppose that
$K^T_{\mu}(z,z_0)\leqslant Q_{z_0}(z)$ a.e. in $D\cap U_{z_0}$ for every point $z_0\in \overline{D}$, a
neighborhood $U_{z_0}$ of $z_0$ and a function $Q_{z_0}:{\Bbb C}\to[0,\infty]$ in the class ${\rm FMO}({z_0})$.
Then the Beltrami equation (\ref{eqBeltrami}) has a regular solution of the Dirichlet problem (\ref{eqGrUsl})
for each continuous function $\varphi:\partial D\to{\Bbb R}$.
\end{theorem}

\medskip

\begin{remark}\label{rm555} In paticular, the conditions and the conclusion of Theorem
\ref{thDIR2fmo} hold if either $Q_{z_0}\in{\rm BMO}_{\rm loc}$ or $Q_{z_0}\in{\rm W}^{1,2}_{\rm loc}$ because
$W^{\,1,2}_{\rm loc} \subset {\rm VMO}_{\rm loc}$, see, e.g., \cite{BN}. \end{remark}

\medskip

Since $K^T_{\mu}(z,z_0) \leqslant K_{\mu}(z)$ for all $z_0\in \Bbb
C$ and $z\in D$, we obtain the following consequence of Theorem
\ref{thDIR2fmo}.

\medskip

\begin{corollary}\label{corDIR2fmo} Let $D$ be a Jordan domain and $\mu:D\to{\Bbb C}$
be a measurable function with $|\mu(z)|<1$ a.e. and such that $K_{\mu}(z)\leqslant Q(z)$ a.e. in $D$ for a
function $Q:{\Bbb C}\to[0,\infty]$ in ${\rm FMO}(\overline{D})$. Then the Beltrami equation (\ref{eqBeltrami})
has a regular solution of the Dirichlet problem (\ref{eqGrUsl}) for each continuous function $\varphi:\partial
D\to{\Bbb R}$. \end{corollary}

\medskip

\begin{corollary}\label{corDIR1000re} In particular, the conclusion of
Theorem \ref{thDIR2fmo} holds if every point $z_0\in\overline{D}$ is the Lebesgue point of a function
$Q_{z_0}:{\Bbb C}\to[0,\infty]$ which is integrable in a neighborhood $U_{z_0}$ of the point $z_0$ such that
$K^T_{\mu}(z,z_0)\leqslant Q_{z_0}(z)$ a.e. in $D\cap U_{z_0}$.
\end{corollary}

\medskip

We assume that $K^T_{\mu}(z,z_0)$ is extended by zero outside of $D$
in the following consequences of Theorem \ref{thDIR2fmo}.

\medskip

\begin{corollary}\label{corDIR1} Let $D$ be a Jordan domain
and $\mu:D\to{\Bbb C}$ be a measurable function with $|\mu(z)|<1$
a.e. such that $K_{\mu}\in L^1_{\rm loc}(D)$ and
\begin{equation}\label{eqDIR6*}\overline{\lim\limits_{\varepsilon\to0}}\quad
\dashint_{B(z_0,\varepsilon)}K^T_{\mu}(z,z_0)\,dm(z)<\infty\qquad\forall\ z_0\in\overline{D}\,.\end{equation}
Then the Beltrami equation (\ref{eqBeltrami}) has a regular solution of the Dirichlet problem (\ref{eqGrUsl})
for each continuous function $\varphi:\partial D\to{\Bbb R}$. \end{corollary}

\medskip

Similarly, choosing in Lemma \ref{lemDIR9} the function
$\psi(t)=1/t$, we come to the following statement.

\medskip

\begin{theorem}\label{thKPRS12b*} Let $D$ be a Jordan domain and $\mu:D\to{\Bbb C}$
be a measurable function with $|\mu(z)|<1$ a.e. such that
$K_{\mu}\in L^1_{\rm loc}(D)$. Suppose that
\begin{equation}\label{eqKPRS12c*}
\int\limits_{\varepsilon<|z-z_0|<\varepsilon_0}K^T_{\mu}(z,z_0)\,\frac{dm(z)}{|z-z_0|^2}
=o\left(\left[\log\frac{1}{\varepsilon}\right]^2\right)\qquad\forall\ z_0\in\overline{D}\end{equation} as
$\varepsilon\to 0$ for some $\varepsilon_0=\delta(z_0)\in(0,d(z_0))$ where $d(z_0)=\sup_{z\in D}|z-z_0|$. Then
the Beltrami equation (\ref{eqBeltrami}) has a regular solution of the Dirichlet problem (\ref{eqGrUsl}) for
each continuous function $\varphi:\partial D\to{\Bbb R}$.\end{theorem}

\medskip

\begin{remark}\label{rmKRRSa*} Choosing in Lemma \ref{lemDIR9} the function
$\psi(t)=1/(t\log{1/t})$ instead of $\psi(t)=1/t$, we are able to
replace (\ref{eqKPRS12c*}) by \begin{equation}\label{eqKPRS12f*}
\int\limits_{\varepsilon<|z-z_0|<\varepsilon_0}\frac{K^T_{\mu}(z,z_0)\,dm(z)}
{\left(|z-z_0|\log{\frac{1}{|z-z_0|}}\right)^2}
=o\left(\left[\log\log\frac{1}{\varepsilon}\right]^2\right)\end{equation}
In general, we are able to give here the whole scale of the
corresponding conditions in $\log$ using functions $\psi(t)$ of the
form
$1/(t\log{1}/{t}\cdot\log\log{1}/{t}\cdot\ldots\cdot\log\ldots\log{1}/{t})$.
\end{remark}

\medskip

Arguing similarly to the proof of Lemma \ref{lemDIR9}, we obtain on the basis of Theorem \ref{th7.KR4.1c}   the
following result.

\medskip

\begin{theorem}\label{thDIR2io} Let $D$ be a Jordan domain in $\Bbb{C}$ and $\mu:D\to{\Bbb C}$
be a measurable function with $|\mu(z)|<1$ a.e. such that $K_{\mu}\in L^1_{\rm loc}(D)$. Suppose that
\begin{equation}\label{eq8.11.2}\int\limits_{0}^{\delta(z_0)}
\frac{dr}{||K^T_{\mu}||_{1}(z_0,r)}=\infty\qquad\forall\
z_0\in\overline{D}\end{equation} for some $\delta(z_0)\in(0,d(z_0))$
where $d(z_0)=\sup\limits_{z\in D}|z-z_0|$ and
\begin{equation}\label{eq8.11.4}
||K^T_{\mu}||_{1}(z_0,r)=\int\limits_{D\cap
S(z_0,r)}K^T_{\mu}(z,z_0)\,|dz|\, .\end{equation} Then the Beltrami
equation (\ref{eqBeltrami}) has a regular solution $f$ of the
Dirichlet problem (\ref{eqGrUsl}) for each continuous function
$\varphi:\partial D\to{\Bbb R}$.\end{theorem}

\medskip

\begin{corollary}\label{corDIR2t} Let $D$ be a Jordan domain
and $\mu:D\to{\Bbb C}$ be a measurable function with $|\mu(z)|<1$
a.e. such that $K_{\mu}\in L^1_{\rm loc}(D)$. Suppose that
\begin{equation}\label{eqDIR61}k_{z_{0}}(\varepsilon)=O\left(\log\frac{1}{\varepsilon}\right)
\qquad\mbox{as}\ \varepsilon\to0\qquad\forall\ z_0\in\overline{D}\end{equation} where $k_{z_0}(\varepsilon)$ is
the average of the function $K^T_{\mu}(z,z_0)$ over $S(z_{0},\varepsilon)$. Then the Beltrami equation
(\ref{eqBeltrami}) has a regular solution of the Dirichlet problem (\ref{eqGrUsl}) for each continuous function
$\varphi:\partial D\to{\Bbb R}$. \end{corollary}

\medskip

\begin{remark}\label{rem2} In particular, the conclusion of Corollary \ref{corDIR2t} holds if
\begin{equation}\label{eqDIR6**} K^T_{\mu}(z,z_0)=O\left(\log\frac{1}{|z-z_0|}\right)\qquad{\rm
as}\quad z\to z_0\quad\forall\ z_0\in\overline{D}\,.\end{equation}\end{remark}

\medskip

\begin{corollary}\label{corDIlogR2t} Let $D$ be a Jordan domain
and $\mu:D\to{\Bbb C}$ be a measurable function with $|\mu(z)|<1$
a.e. such that $K_{\mu}\in L^1_{\rm loc}(D)$ and
\begin{equation}\label{edgddgsddgdsfweweIR6**c}
k_{z_{0}}(\varepsilon)=O\left(\left[\log\frac{1}{\varepsilon}\cdot\log\log\frac{1}
{\varepsilon}\cdot\ldots\cdot\log\ldots\log\frac{1}{\varepsilon}
\right]\right) \qquad\forall\ z_0\in \overline{D}
\end{equation}
as $\varepsilon\to0$, where $k_{z_0}(\varepsilon)$ is the average of
the function $K^T_{\mu}(z,z_0)$ over $S(z_{0},\varepsilon)$. Then
the Beltrami equation (\ref{eqBeltrami}) has a regular solution of
the Dirichlet problem (\ref{eqGrUsl}) for each continuous function
$\varphi:\partial D\to{\Bbb R}$. \end{corollary}

\medskip

Finally, combining Theorems \ref{th5.555} and \ref{thDIR2io}, we
obtain the following.

\medskip

\begin{theorem}\label{thKR4.1} Let $D$ be a Jordan domain and
$\mu:D\to{\Bbb C}$ be a measurable function with $|\mu(z)|<1$ a.e.
such that $K_{\mu}\in L^1_{\rm loc}(D)$. Suppose that
\begin{equation}\label{eqKR4.1sh}\int\limits_{D\cap U_{z_0}}\Phi_{z_0}\left(K^T_{\mu}(z,z_0)\right)\,dm(z)<\infty
\qquad\forall\ z_0\in \overline{D}\end{equation} for a neighborhood
$U_{z_0}$ of $z_0$ and a convex non-decreasing function
$\Phi_{z_0}:[0,\infty]\to[0,\infty]$ such that
\begin{equation}\label{eqKR4.2ras}
\int\limits_{\delta(z_0)}^{\infty}\frac{d\tau}{\tau\Phi_{z_0}^{-1}(\tau)}=\infty\end{equation} for some
$\delta(z_0)>\Phi_{z_0}(0)$. Then the Beltrami equation (\ref{eqBeltrami}) has a regular solution of the
Dirichlet problem (\ref{eqGrUsl}) for each continuous function $\varphi:\partial D\to{\Bbb R}$. \end{theorem}

\medskip

\begin{corollary}\label{corDIR1000} In particular, the conclusion of
Theorem \ref{thKR4.1} holds if
\begin{equation}\label{eqKR4.1c}\int\limits_{D\cap
U_{z_0}}e^{\alpha(z_0) K^T_{\mu}(z,z_0)}\,dm(z)<\infty \qquad\forall\ z_0\in \overline{D}\end{equation} for some
$\alpha(z_0)>0$ and a neighborhood $U_{z_0}$ of the point $z_0$.
\end{corollary}

\medskip

Since $K^T_{\mu}(z,z_0) \leqslant K_{\mu}(z)$ for all $z_0\in \Bbb
C$ and $z\in D$, we obtain the following consequence of Theorem
\ref{thKR4.1}.

\medskip

\begin{corollary}\label{corc.KR4.1} Let $D$ be a Jordan domain and
$\mu:D\to{\Bbb C}$ be a measurable function with $|\mu(z)|<1$ a.e.
such that
\begin{equation}\label{eqc.KR4.1sh}\int\limits_{D}\Phi\left(K_{\mu}(z)\right)\,dm(z)<\infty\end{equation}
for a convex non-decreasing function $\Phi:[0,\infty]\to[0,\infty]$.
If
\begin{equation}\label{eqc.KR4.2ras}
\int\limits_{\delta}^{\infty}\frac{d\tau}{\tau\Phi^{-1}(\tau)}=\infty\end{equation} for some $\delta>\Phi(0)$.
Then the Beltrami equation (\ref{eqBeltrami}) has a regular solution of the Dirichlet problem (\ref{eqGrUsl})
for each continuous function $\varphi:\partial D\to{\Bbb R}$. \end{corollary}

\medskip

\begin{remark}\label{remeq333F_2} By the Stoilow theorem, see, e.g., \cite{Sto}, a regular solution $f$
of the Dirichlet problem (\ref{eqGrUsl}) for the Beltrami equation
(\ref{eqBeltrami}) with $K_{\mu}\in L^1_{\rm loc}(D)$ can be
represented in the form $f=h\circ F$ where $h$ is an analytic
function and $F$ is a homeomorphic regular solution of
(\ref{eqBeltrami}) in the class $W_{\rm loc}^{1,1}$. Thus, by
Theorem 5.1 in \cite{RSY13} the condition (\ref{eqc.KR4.2ras}) is
not only sufficient but also necessary to have a regular solution of
the Dirichlet problem (\ref{eqGrUsl}) for an arbitrary Beltrami
equation (\ref{eqBeltrami}) with the integral constraints
(\ref{eqc.KR4.1sh}) for any non-constant continuous  function
$\varphi:\partial D\to\Bbb{R}$, see also Remark \ref{remeq333F}.
\end{remark}

\medskip

\section{On pseudoregular solutions in multiply connected domains}

As it was first noted by Bojarski, see, e.g., section 6 of Chapter 4 in \cite{Vekua}, that in the case of
multiply connected domains the Dirichlet problem for the Beltrami equation, generally speaking, has no solutions
in the class of continuous (simply-valued) functions. Hence it is arose the question: whether the existence of
solutions for the Dirichlet problem can be obtained in a wider function class for the case? It is turned out to
be that this is possible in the class of functions having a certain number of poles at prescribed points in $D$.
More precisely, for $\varphi(\zeta)\not\equiv{\rm const}$, a {\bf pseudoregular solution} of the problem is a
continuous (in $\overline{\Bbb C}={\Bbb C}\cup\{\infty\}$) discrete open mapping $f:D\to\overline{\Bbb C}$ in
the class $W^{1,1}_{\rm loc}$ (outside of these poles) with the Jacobian $J_{f}(z)=|f_z|^2-|f_{\bar z}|^2\neq0$
a.e. satisfying (\ref{eqBeltrami}) a.e. and condition (\ref{eqGrUsl}).

\medskip

As above, we assume in the following lemma that $K^T_{\mu}(z,z_0)$
is extended by zero outside of the domain $D$.

\medskip

\begin{lemma}\label{lem13.5.333ps} Let $D$ be a bounded domain in
${\Bbb C}$ whose boundary consists of $n\geqslant2$ mutually disjoint Jordan curves. Suppose that $\mu:D\to{\Bbb
C}$ is a measurable function with $|\mu(z)|<1$ a.e., $K_{\mu}(z)\in L^1_{\rm loc}(D),$ $K^T_{\mu}(z,z_0)$ is
integrable over $D\bigcap S (z_0, r)$ for a.e. $r \in (0, \varepsilon_0)$ and
\begin{equation}\label{p.3omal}
\int\limits_{\varepsilon<|z-z_0|<\varepsilon_0}
K^T_{\mu}(z,z_0)\cdot\psi^2_{z_0,\varepsilon}(|z-z_0|)\,dm(z)=o(I_{z_0}^{2}(\varepsilon))\quad{\rm
as}\quad\varepsilon\to0\ \ \forall\ z_0\in\overline{D}\end{equation} for some $\varepsilon_0\in(0,\delta_0)$
where $\delta_0=\delta(z_0)=\sup\limits_{z\in D}|z-z_0|$ and $\psi_{z_0,\varepsilon}(t)$ is a family of
non-negative measurable (by Lebesgue) functions on $(0,\infty)$ such that
\begin{equation}\label{eqp.3.5.3}
I_{z_0}(\varepsilon)\colon =\int\limits_{\varepsilon}^{\varepsilon_0}
\psi_{z_0,\varepsilon}(t)\,dt<\infty\qquad\forall\ \varepsilon\in(0,\varepsilon_0)\,.\end{equation} Then the
Beltrami equation (\ref{eqBeltrami}) has a pseudoregular solution of the Dirichlet problem (\ref{eqGrUsl}) for
each continuous function $\varphi:\partial D\to{\Bbb R}$, $\varphi(\zeta)\not\equiv{\rm const}$, with poles at
$n$ prescribed points $z_i\in D$, $i=1,\ldots,n$. \end{lemma}

\medskip

\begin{proof} Let $F$ be a regular homeomorphic solution of the
Beltrami equation (\ref{eqBeltrami}) of the class $W^{1,1}_{\rm loc}$ that
exists by Lemma 4.1 in \cite{RSY$_6$}. Consider $D^*=f(D)$.
Note that $\partial D^*$ has $n$ connected components $\Gamma_i$, $i=1,\ldots,n$
that correspond in the natural way to connected components of $\partial D$,
the Jordan curves $\gamma_i$, see, e.g., either Lemma 5.3 in \cite{IR}
or Lemma 6.5 in \cite{MRSY}.

Thus, by Theorem V.6.2 in \cite{Gol} the domain $D_*$ can be mapped
by a conformal map $R$ onto a circular domain ${\Bbb D}_*$ whose
boundary consists of $n$ circles or points, i.e. ${\Bbb D}_*$ has a
weakly flat boundary. Note that the mapping $g:=R\circ F$ is a
regular homeomorphic solution of the Beltrami equation in  the class
$W_{\rm loc}^{1,1}$ admitting a homeomorphic extension
$g_*:\overline{D}\to\overline{\Bbb D_*}$ by Lemma \ref{lem3.3.333}.

Let us find a solution of the Dirichlet problem (\ref{eqGrUsl}) in
the form $f=h\circ g$ where $h$ is a meromorphic function with $n$
poles at the prescribed points $w_i=g(z_i)$, $i=1,\ldots,n$ in
${\Bbb D}_*$ with the boundary condition
$$\lim\limits_{w\to\zeta}{\rm Re}\,h(w)=\varphi(g_{*}^{-1}(\zeta))
\qquad\forall\ \zeta\in\partial{\Bbb D}_* \,.$$ Such a function $h$ exists by Theorem 4.14 in \cite{Vekua}.

We see that the function $f=h\circ g$ is the desired pseudoregular
solution of the Dirichlet problem (\ref{eqGrUsl}) for the Beltrami
equation (\ref{eqBeltrami}) with $n$ poles  just at these
prescribed points $z_i$, $i=1,\ldots,n$.
\end{proof}

\medskip

Arguing similarly to the last section, by the special choice of the
functional parameter $\psi$ in Lemma \ref{lem13.5.333ps}, we obtain
the following results.

\medskip

\begin{theorem}\label{thDIR2} Let $D$ be a bounded domain in
${\Bbb C}$ whose boundary consists of $n\geqslant2$ mutually disjoint Jordan curves and $\mu:D\to{\Bbb C}$ be a
measurable function with $|\mu(z)|<1$ a.e. such that $K_{\mu}\in L^1_{\rm loc}(D)$. Suppose that, for every
point $z_0\in \overline{D}$ and a neighborhood $U_{z_0}$ of $z_0$, $K^T_{\mu}(z,z_0)\leqslant Q_{z_0}(z)$ a.e.
in $D\cap U_{z_0}$ for a function $Q_{z_0}:{\Bbb C}\to[0,\infty]$ in the class ${\rm FMO}({z_0})$. Then the
Beltrami equation (\ref{eqBeltrami}) has a pseudoregular solution of the Dirichlet problem (\ref{eqGrUsl}) for
every continuous function $\varphi:\partial D\to{\Bbb R}$, $\varphi(\zeta)\not\equiv{\rm const}$, with poles at
$n$ prescribed points in $D$. \end{theorem}

\medskip

\begin{remark}\label{rmp.555} In paticular, the conditions and the conclusion of Theorem
\ref{thDIR2} hold if either $Q_{z_0}\in{\rm BMO}_{\rm loc}$ or $Q_{z_0}\in{\rm W}^{1,2}_{\rm loc}$ because
$W^{\,1,2}_{\rm loc} \subset {\rm VMO}_{\rm loc}$. \end{remark}

\medskip

\begin{corollary}\label{corc.DIR2} Let $D$ be a bounded domain in
${\Bbb C}$ whose boundary consists of $n\geqslant2$ mutually disjoint Jordan curves and $\mu:D\to{\Bbb C}$ be a
measurable function with $|\mu(z)|<1$ a.e. such that $K_{\mu}(z)\leqslant Q(z)$ a.e. in $\overline{D}$ for a
function $Q:{\Bbb C}\to[0,\infty]$ in the class ${\rm FMO}(\overline{D})$. Then the Beltrami equation
(\ref{eqBeltrami}) has a pseudoregular solution of the Dirichlet problem (\ref{eqGrUsl}) for every continuous
function $\varphi:\partial D\to{\Bbb R}$, $\varphi(\zeta)\not\equiv{\rm const}$, with poles at $n$ prescribed
points in $D$. \end{corollary}

\medskip

\begin{corollary}\label{corDIR1500re} In particular, the conclusion of
Theorem \ref{thDIR2} holds if every point $z_0\in\overline{D}$ is the Lebesgue point of a function
$Q_{z_0}:{\Bbb C}\to[0,\infty]$ which is integrable in a neighborhood $U_{z_0}$ of the point $z_0$ such that
$K^T_{\mu}(z,z_0)\leqslant Q_{z_0}(z)$ a.e. in $D\cap U_{z_0}$.\end{corollary}

\medskip

\begin{corollary}\label{corDIR1000} In particular, the conclusion of
Theorem \ref{thDIR2} holds if \begin{equation}\label{eqDIR6*} \overline{\lim\limits_{\varepsilon\to0}}\ \ \
\dashint_{B(z_0,\varepsilon)}K^T_{\mu}(z,z_0)\,dm(z)<\infty\qquad\forall\ z_0\in\overline{D}\,.\end{equation}
\end{corollary}

\medskip

As above, here we assume that $K^T_{\mu}(z,z_0)$ is extended by zero outside of $D$.

\medskip

\begin{theorem}\label{thp.KPRS12b*} Let $D$ be a bounded domain in
${\Bbb C}$ whose boundary consists of $n\geqslant2$ mutually
disjoint Jordan curves and $\mu:D\to{\Bbb C}$ be a measurable
function with $|\mu(z)|<1$ a.e. such that $K_{\mu}\in L^1_{\rm
loc}(D)$. Suppose that
\begin{equation}\label{eqp.KPRS12c*}
\int\limits_{\varepsilon<|z-z_0|<\varepsilon_0}K^T_{\mu}(z,z_0)\,\frac{dm(z)}{|z-z_0|^2}
=o\left(\left[\log\frac{1}{\varepsilon}\right]^2\right)\qquad\forall\ z_0\in\overline{D}\end{equation} as
$\varepsilon\to 0$ for some $\varepsilon_0=\delta(z_0)\in(0,d(z_0))$ where $d(z_0)=\sup_{z\in D}|z-z_0|$. Then
the Beltrami equation (\ref{eqBeltrami}) has a pseudoregular solution of the Dirichlet problem (\ref{eqGrUsl})
for every continuous function $\varphi:\partial D\to{\Bbb R}$, $\varphi(\zeta)\not\equiv{\rm const}$, with poles
at $n$ prescribed points in $D$.\end{theorem}

\medskip

\begin{remark}\label{rmp.KRRSa*} Choosing in Lemma \ref{lem13.5.333ps} the function
$\psi(t)=1/(t\log{1/t})$ instead of $\psi(t)=1/t$, we are able to
replace (\ref{eqp.KPRS12c*}) by \begin{equation}\label{eqp.KPRS12f*}
\int\limits_{\varepsilon<|z-z_0|<\varepsilon_0}\frac{K^T_{\mu}(z,z_0)\,dm(z)}
{\left(|z-z_0|\log{\frac{1}{|z-z_0|}}\right)^2}
=o\left(\left[\log\log\frac{1}{\varepsilon}\right]^2\right)\end{equation}
In general, we are able to give here the whole scale of the
corresponding conditions in $\log$ using functions $\psi(t)$ of the
form
$1/(t\log{1}/{t}\cdot\log\log{1}/{t}\cdot\ldots\cdot\log\ldots\log{1}/{t})$.
\end{remark}

\medskip

Arguing similarly to the proof of Lemma \ref{lem13.5.333ps}, we
obtain on the basis of the Theorem \ref{th7.KR4.1c} the following.

\medskip

\begin{theorem}\label{thDIR2i} Let $D$  be a bounded  domain in $\Bbb{C}$ whose
boundary consists of $n\geqslant2$ mutually disjoint Jordan curves
and $\mu:D\to{\Bbb C}$ be a measurable function with $|\mu(z)|<1$
a.e. such that $K_{\mu}\in L^1_{\rm loc}(D)$. Suppose that
\begin{equation}\label{eq4.8.11.2}\int\limits_{0}^{\delta(z_0)}
\frac{dr}{||K^T_{\mu}||_{1}(z_0,r)}=\infty\qquad\forall\
z_0\in\overline{D}\end{equation} for some $\delta(z_0)\in(0,d(z_0))$
where $d(z_0)=\sup\limits_{z\in D}|z-z_0|$ and
\begin{equation}\label{eq8.11.4}
||K^T_{\mu}||_{1}(z_0,r)=\int\limits_{D\cap S(z_0,r)}K^T_{\mu}(z,z_0)\,|dz|\,.\end{equation} Then the Beltrami
equation (\ref{eqBeltrami}) has a pseudoregular solution of the Dirichlet problem (\ref{eqGrUsl}) for every
continuous function $\varphi:\partial D\to{\Bbb R}$, $\varphi(\zeta)\not\equiv{\rm const}$, with poles at $n$
prescribed points in $D$. \end{theorem}

\medskip

\begin{corollary}\label{corp.DIR2t} Let $D$ be a bounded domain in
${\Bbb C}$ whose boundary consists of $n\geqslant2$ mutually disjoint Jordan curves and $\mu:D\to{\Bbb C}$ be a
measurable function with $|\mu(z)|<1$ a.e. such that $K_{\mu}\in L^1_{\rm loc}(D)$. Suppose that
\begin{equation}\label{eqp.DIR61fa} k_{z_0}(\varepsilon)=O\left(\log\frac{1}{\varepsilon}\right)\quad
\text{as}\quad\varepsilon\to0\quad\forall\ z_0\in\overline{D}\,,\end{equation} where $k_{z_0}(\varepsilon)$ is
the average of the function $K^T_{\mu}(z,z_0)$ over the circle $\{z\in{\Bbb C}:|z-z_0|=\varepsilon\}$. Then the
Beltrami equation (\ref{eqBeltrami}) has a pseudoregular solution of the Dirichlet problem (\ref{eqGrUsl}) for
every continuous function $\varphi:\partial D\to{\Bbb R}$, $\varphi(\zeta)\not\equiv{\rm const}$, with poles at
$n$ prescribed points in $D$.\end{corollary}

\medskip

\begin{remark}\label{rem2tyyr} In particular, the conclusion of Corollary \ref{corp.DIR2t} holds if
\begin{equation}\label{eqDIR6w**}K^T_{\mu}(z,z_0)=O\left(\log\frac{1}{|z-z_0|}\right)\qquad{\rm
as}\quad z\to z_0\quad\forall\ z_0\in\overline{D}\,.\end{equation}
\end{remark}

\medskip

\begin{corollary}\label{corDIdgdslogR2t} Let $D$ be a bounded domain in
${\Bbb C}$ whose boundary consists of $n\geqslant2$ mutually
disjoint Jordan curves and $\mu:D\to{\Bbb C}$ be a measurable
function with $|\mu(z)|<1$ a.e. such that $K_{\mu}(z)\in L^1_{\rm
loc}(D)$. Suppose that
\begin{equation}\label{edgddgsddloggdsfweweIR6**c}
k_{z_{0}}(\varepsilon)=O\left(\left[\log\frac{1}{\varepsilon}\cdot\log\log\frac{1}{\varepsilon}\cdot\ldots\cdot\log\ldots\log\frac{1}{\varepsilon}
\right]\right) \qquad\forall\ z_0\in \overline{D}
\end{equation}
where $k_{z_0}(\varepsilon)$ is the average of the function
$K^T_{\mu}(z,z_0)$ over the circle $\{z\in{\Bbb
C}:|z-z_0|=\varepsilon\}$. Then the Beltrami equation
(\ref{eqBeltrami}) has a pseudoregular solution of the Dirichlet
problem (\ref{eqGrUsl}) for every continuous function
$\varphi:\partial D\to{\Bbb R}$, $\varphi(\zeta)\not\equiv{\rm
const}$, with poles at $n$ prescribed points in $D$.\end{corollary}

\medskip

\begin{theorem}\label{thq13.5.333} Let $D$ be a bounded domain in
${\Bbb C}$ whose boundary consists of $n\geqslant2$ mutually disjoint Jordan curves and let $\mu:D\to{\Bbb C}$
be a measurable function with $|\mu(z)|<1$ a.e. such that $K_{\mu}\in L^1_{\rm loc}(D)$. Suppose that
\begin{equation}\label{eqKR4.1}\int\limits_{D\cap U_{z_0}}\Phi_{z_0}\left(K^T_{\mu}(z,z_0)\right)\,dm(z)<\infty
\quad\forall\ z_0\in\overline{D}\end{equation} for a neighborhood $U_{z_0}$of the point $z_0$ and a convex
non-decreasing function $\Phi_{z_0}:[0,\infty]\to[0,\infty]$ such that
\begin{equation}\label{eqKR4.2}\int\limits_{\delta(z_0)}^{\infty}\frac{d\tau}{\tau\Phi_{z_0}^{-1}(\tau)}=\infty\end{equation}
for some $\delta(z_0)>\Phi_{z_0}(0)$. Then the Beltrami equation (\ref{eqBeltrami}) has a pseudoregular solution
of the Dirichlet problem (\ref{eqGrUsl}) for each continuous function $\varphi:\partial D\to{\Bbb R}$,
$\varphi(\zeta)\not\equiv{\rm const}$, with poles at $n$ prescribed inner points of $D$.
\end{theorem}

\medskip

\begin{corollary}\label{corDIR1000} In particular, the conclusion of
Theorem \ref{thq13.5.333} holds if
\begin{equation}\label{eqp.KR4.1c}\int\limits_{D\cap
U_{z_0}}e^{\alpha(z_0) K^T_{\mu}(z,z_0)}\,dm(z)<\infty \qquad\forall\ z_0\in \overline{D}\end{equation} for some
$\alpha(z_0)>0$ and a neighborhood $U_{z_0}$ of the point $z_0$.\end{corollary}

\medskip

\begin{corollary}\label{corc.q13.5.333} Let $D$ be a bounded domain in
${\Bbb C}$ whose boundary consists of $n\geqslant2$ mutually
disjoint Jordan curves and let $\mu:D\to{\Bbb C}$ be a measurable
function with $|\mu(z)|<1$ a.e. such that
\begin{equation}\label{eqc.KR4.1}\int\limits_{D}\Phi\left(K_{\mu}(z)\right)\,dm(z)<\infty\end{equation}
for a convex non-decreasing function $\Phi:[0,\infty]\to[0,\infty]$.
If
\begin{equation}\label{eqc.KR4.2}\int\limits_{\delta}^{\infty}\frac{d\tau}{\tau\Phi^{-1}(\tau)}=\infty\end{equation}
for some $\delta>\Phi(0)$. Then the Beltrami equation (\ref{eqBeltrami}) has a pseudoregular solution of the
Dirichlet problem (\ref{eqGrUsl}) for each continuous function $\varphi:\partial D\to{\Bbb R}$,
$\varphi(\zeta)\not\equiv{\rm const}$, with poles at $n$ prescribed inner points in $D$.
\end{corollary}

\medskip

\section{On multi-valued solutions in finitely connected domains}

In  finitely connected domains $D$ in $\Bbb{C}$, in addition to pseudoregular solutions, the Dirichlet problem
(\ref{eqGrUsl}) for the Beltrami equation (\ref{eqBeltrami}) admits multi-valued solutions in the spirit of the
theory of multi-valued analytic functions. We say that a discrete open mapping $f:B(z_0,\varepsilon_0)\to{\Bbb
C}$, where $B(z_0,\varepsilon_0)\subseteq D$, is a {\bf local regular solution of the equation}
(\ref{eqBeltrami}) if $f\in W_{\rm loc}^{1,1}$, $J_f(z)\neq0$ and $f$ satisfies (\ref{eqBeltrami}) a.e. in
$B(z_0,\varepsilon_0)$.

\medskip

The local regular solutions $f:B(z_0,\varepsilon_0)\to{\Bbb C}$ and
$f_*:B(z_*,\varepsilon_*)\to{\Bbb C}$ of the equation
(\ref{eqBeltrami}) will be called extension of each to other if
there is a finite chain of such solutions
$f_i:B(z_i,\varepsilon_i)\to\Bbb{C}$, $i=1,\ldots,m$, that
$f_1=f_0$, $f_m=f_*$ and $f_i(z)\equiv f_{i+1}(z)$ for $z\in
E_i:=B(z_i,\varepsilon_i)\cap
B(z_{i+1},\varepsilon_{i+1})\neq\emptyset$, $i=1,\ldots,m-1$. A
collection of local regular solutions
$f_j:B(z_j,\varepsilon_j)\to{\Bbb C}$, $j\in J$, will be called a
{\bf multi-valued solution} of the equation (\ref{eqBeltrami}) in
$D$ if the disks $B(z_j,\varepsilon_j)$ cover the whole domain $D$
and $f_j$ are extensions of each to other through the collection and
the collection is maximal by inclusion. A multi-valued solution of
the equation (\ref{eqBeltrami}) will be called a {\bf multi-valued
solution of the Dirichlet problem} (\ref{eqGrUsl}) if $u(z)={\rm
Re}\,f(z)={\rm Re}\,f_{j}(z)$, $z\in B(z_j,\varepsilon_j)$, $j\in
J$, is a simply-valued function in $D$ satisfying the condition
$\lim\limits_{z\in\zeta}u(z)=\varphi(\zeta)$ for all $\zeta\to
\partial D$.

\medskip

As usual, we assume in the following lemma that $K^T_{\mu}(\cdot,z_0)$ is extended by zero outside of the domain
$D$.

\medskip

\begin{lemma}\label{lem13.5.3334} Let $D$ be a bounded domain in
${\Bbb C}$ whose boundary consists of $n\geqslant2$ mutually disjoint Jordan curves. Suppose that $\mu:D\to{\Bbb
C}$ is a measurable function with $|\mu(z)|<1$ a.e., $K_{\mu}\in L^1_{\rm loc}(D),$ $K^T_{\mu}(z,z_0)$ is
integrable over $D\bigcap S (z_0, r)$ for a.e. $r \in (0, \varepsilon_0)$ and
\begin{equation}\label{m.3omal}
\int\limits_{\varepsilon<|z-z_0|<\varepsilon_0}
K^T_{\mu}(z,z_0)\cdot\psi^2_{z_0,\varepsilon}(|z-z_0|)\,dm(z)=o(I_{z_0}^{2}(\varepsilon))\quad{\rm
as}\quad\varepsilon\to0\ \ \forall\ z_0\in\overline{D}\end{equation}
for some $\varepsilon_0\in(0,\delta_0)$ where
$\delta_0=\delta(z_0)=\sup_{z\in D}|z-z_0|$ and
$\psi_{z_0,\varepsilon}(t)$ is a family of non-negative measurable
(by Lebesgue) functions on $(0,\infty)$ such that
\begin{equation}\label{eqm.3.5.3}
I_{z_0}(\varepsilon)\colon =\int\limits_{\varepsilon}^{\varepsilon_0}
\psi_{z_0,\varepsilon}(t)\,dt<\infty\qquad\forall\ \varepsilon\in(0,\varepsilon_0)\,.\end{equation} Then the
Beltrami equation (\ref{eqBeltrami}) has a multi-valued solutions of the Dirichlet problem (\ref{eqGrUsl}) for
each continuous function $\varphi:\partial D\to{\Bbb R}$. \end{lemma}

\medskip

\begin{proof} Let $F$ be a regular homeomorphic solution of the
Beltrami equation (\ref{eqBeltrami}) in the class $W_{\rm
loc}^{1,1}$ that exists by Lemma 4.1 in \cite{RSY$_6$}. As it was
showed under the proof of Lemma \ref{lem13.5.333ps}, we may assume
that $D_*:=F(D)$ is a circular domain and that $F$ can be extended
to a homeomorphism $F_*:\overline{D}\to\overline{D_*}$. Let
$u:D_*\to\Bbb{R}$ be a harmonic function such that
$$\lim\limits_{w\to\zeta}u(w)=\varphi (F_{*}^{-1}(\zeta))\qquad\forall\ \zeta\in\partial D^*$$
whose existence is well-known, see, e.g., Section 3 of Chapter VI
in \cite{Gol}.

Let $z_0\in D$, $B_0:=B(z_0,\varepsilon_0)\subseteq D$ for some
$\varepsilon_0>0$. Then the domain $D_0=F(B_0)$ is simply connected
and hence there is a harmonic function $v(w)$ such that
$h(w)=u(w)+iv(w)$ is a holomorphic function which is unique up to an
additive constant, see, e.g., Theorem 1 in Section 7 of Chapter III,
Part 3 in \cite{HuCo}. Note that $f_0:=h\circ F|_{B_0}$ is a local
regular solution of the Beltrami equation (\ref{eqBeltrami}). Note
that the function $h$ can be extended to a multi-valued analytic
function $H$ in the domain $D_*$ and, thus, $H\circ F$ gives the
desired multi-valued solutions of the Dirichlet problem
(\ref{eqGrUsl}) for the Beltrami equation (\ref{eqBeltrami}).
\end{proof}

\medskip

In particular, by Lemma \ref{lem13.5.3334} above and Lemma
\ref{lem13.4.2} we obtain the following.

\medskip

\begin{theorem}\label{12thDIR2fmo} Let $D$ be a bounded domain in
${\Bbb C}$ whose boundary consists of $n\geqslant2$ mutually disjoint Jordan curves and $\mu:D\to{\Bbb C}$ be a
measurable function with $|\mu(z)|<1$ a.e. such that $K_{\mu}\in L^1_{\rm loc}(D)$. Suppose that, for every
point $z_0\in \overline{D}$ and a neighborhood $U_{z_0}$ of $z_0$, $K^T_{\mu}(z,z_0)\leqslant Q_{z_0}(z)$ a.e.
in $D\cap U_{z_0}$ for a function $Q_{z_0}:{\Bbb C}\to[0,\infty]$ in the class ${\rm FMO}({z_0})$. Then the
Beltrami equation (\ref{eqBeltrami}) has a multi-valued solutions of the Dirichlet problem (\ref{eqGrUsl}) for
each continuous function $\varphi:\partial D\to{\Bbb R}$.\end{theorem}

\medskip

\begin{remark}\label{rmm.555} In paticular, the conditions and the conclusion of Theorem
\ref{12thDIR2fmo} hold if either $Q_{z_0}\in{\rm BMO}_{\rm loc}$ or $Q_{z_0}\in{\rm W}^{1,2}_{\rm loc}$ because
$W^{\,1,2}_{\rm loc} \subset {\rm VMO}_{\rm loc}$. \end{remark}

\medskip

\begin{corollary}\label{corc.12thDIR2fmo} Let $D$ be a bounded domain in
${\Bbb C}$ whose boundary consists of $n\geqslant2$ mutually disjoint Jordan curves and $\mu:D\to{\Bbb C}$ be a
measurable function with $|\mu(z)|<1$ a.e. and such that $K_{\mu}(z)\leqslant Q(z)$ a.e. in $D$ for a function
$Q:\Bbb{C}\to[0,\infty]$ in ${\rm FMO}(\overline{D})$. Then the Beltrami equation (\ref{eqBeltrami}) has a
multi-valued solutions of the Dirichlet problem (\ref{eqGrUsl}) for each continuous function $\varphi:\partial
D\to{\Bbb R}$.\end{corollary}

\medskip

\begin{corollary}\label{12corDIR1500re} In particular, the conclusion of
Theorem \ref{12thDIR2fmo} holds if every point $z_0\in\overline{D}$ is Lebesgue point of a function
$Q_{z_0}:{\Bbb C}\to[0,\infty]$ which is integrable in a neighborhood $U_{z_0}$ of the point $z_0$ such that
$K^T_{\mu}(z,z_0)\leqslant Q_{z_0}(z)$ a.e. in $D\cap U_{z_0}$.
\end{corollary}

\medskip

We assume that $K^T_{\mu}(z,z_0)$ is extended by zero outside of $D$
in the next con\-se\-quen\-ces of Theorem \ref{12thDIR2fmo}.

\medskip

\begin{corollary}\label{12corDIR1} Let $D$ be a bounded domain in
${\Bbb C}$ whose boundary consists of $n\geqslant2$ mutually disjoint Jordan curves and $\mu:D\to{\Bbb C}$ be a
measurable function with $|\mu(z)|<1$ a.e. such that $K_{\mu}\in L^1_{\rm loc}(D)$. Suppose that
\begin{equation}\label{eqDIR6*}\overline{\lim\limits_{\varepsilon\to0}}\quad
\dashint_{B(z_0,\varepsilon)}K^T_{\mu}(z,z_0)\,dm(z)<\infty\qquad\forall\ z_0\in\overline{D}\,.\end{equation}
Then the Beltrami equation (\ref{eqBeltrami}) has a multi-valued solutions of the Dirichlet problem
(\ref{eqGrUsl}) for each continuous function $\varphi:\partial D\to{\Bbb R}$. \end{corollary}

\medskip

\begin{theorem}\label{thm.KPRS12b*} Let $D$ be a bounded domain in
${\Bbb C}$ whose boundary consists of $n\geqslant2$ mutually
disjoint Jordan curves and $\mu:D\to{\Bbb C}$ be a measurable
function with $|\mu(z)|<1$ a.e. such that $K_{\mu}\in L^1_{\rm
loc}(D)$. Suppose that
\begin{equation}\label{eqm.KPRS12c*}
\int\limits_{\varepsilon<|z-z_0|<\varepsilon_0}K^T_{\mu}(z,z_0)\,\frac{dm(z)}{|z-z_0|^2}
=o\left(\left[\log\frac{1}{\varepsilon}\right]^2\right)\qquad\forall\ z_0\in\overline{D}\end{equation} as
$\varepsilon\to 0$ for some $\varepsilon_0=\delta(z_0)\in(0,d(z_0))$ where $d(z_0)=\sup_{z\in D}|z-z_0|$. Then
the Beltrami equation (\ref{eqBeltrami}) has a pseudoregular solution of the Dirichlet problem (\ref{eqGrUsl})
for every continuous function $\varphi:\partial D\to{\Bbb R}$, $\varphi(\zeta)\not\equiv{\rm const}$, with poles
at $n$ prescribed points in $D$.\end{theorem}

\medskip

\begin{remark}\label{rmm.KRRSa*} Choosing in Lemma \ref{lem13.5.3334} the function
$\psi(t)=1/(t\log{1/t})$ instead of $\psi(t)=1/t$, we are able to
replace (\ref{eqm.KPRS12c*}) by \begin{equation}\label{eqm.KPRS12f*}
\int\limits_{\varepsilon<|z-z_0|<\varepsilon_0}\frac{K^T_{\mu}(z,z_0)\,dm(z)}
{\left(|z-z_0|\log{\frac{1}{|z-z_0|}}\right)^2}
=o\left(\left[\log\log\frac{1}{\varepsilon}\right]^2\right)\end{equation}
In general, we are able to give here the whole scale of the
corresponding conditions in $\log$ using functions $\psi(t)$ of the
form
$1/(t\log{1}/{t}\cdot\log\log{1}/{t}\cdot\ldots\cdot\log\ldots\log{1}/{t})$.
\end{remark}

\medskip

Arguing similarly to the proof of Lemma \ref{lem13.5.3334}, we
obtain on the basis of Theorem \ref{th7.KR4.1c} the next result.

\begin{theorem}\label{12thDIR2io} Let $D$ be a bounded domain in
${\Bbb C}$ whose boundary consists of $n\geqslant2$ mutually disjoint Jordan curves and $\mu:D\to\Bbb{C}$ be a
measurable function with $|\mu(z)|<1$ a.e. such that $K_{\mu}\in L^1_{\rm loc}(D)$. Suppose that
\begin{equation}\label{eq3.8.11.2}\int\limits_{0}^{\delta(z_0)}
\frac{dr}{||K^T_{\mu}||_{1}(z_0,r)}=\infty \qquad\forall\
z_0\in\overline{D}\end{equation} for some $\delta(z_0)\in(0,d(z_0))$
where $d(z_0)=\sup\limits_{z\in D}|z-z_0|$ and
\begin{equation}\label{eq8.11.4}
||K^T_{\mu}||_{1}(z_0,r)=\int\limits_{D\cap S(z_0,r)}K^T_{\mu}(z,z_0)\,|dz|\, .\end{equation}  Then the Beltrami
equation (\ref{eqBeltrami}) has a multi-valued solutions of the Dirichlet problem (\ref{eqGrUsl}) for each
continuous function $\varphi:\partial D\to{\Bbb R}$.\end{theorem}

\medskip

\begin{corollary}\label{m.12corDIR2t} Let $D$ be a bounded domain in
${\Bbb C}$ whose boundary consists of $n\geqslant2$ mutually disjoint Jordan curves and $\mu:D\to{\Bbb C}$ be a
measurable function with $|\mu(z)|<1$ a.e. such that $K_{\mu}\in L^1_{\rm loc}(D)$. Suppose that
\begin{equation}\label{eqm.DIR61}
k_{z_0}(\varepsilon)=O\left(\log\frac{1}{\varepsilon}\right)\quad \text{as}\quad\varepsilon\to0\quad\forall\
z_0\in\overline{D}\end{equation} where $k_{z_0}(\varepsilon)$ is the average of the function $K^T_{\mu}(z,z_0)$
over $S(z_0,\varepsilon)$. Then the Beltrami equation (\ref{eqBeltrami}) has a multi-valued solutions of the
Dirichlet problem (\ref{eqGrUsl}) for each continuous function $\varphi:\partial D\to{\Bbb R}$.\end{corollary}

\medskip

\begin{remark}\label{rem2} In particular, the conclusion of Corollary \ref{m.12corDIR2t} holds if
\begin{equation}\label{eqDIR6**}K^T_{\mu}(z,z_0)=O\left(\log\frac{1}{|z-z_0|}\right)\qquad{\rm
as}\quad z\to z_0\quad\forall\ z_0\in\overline{D}\,.\end{equation}\end{remark}

\medskip

\begin{corollary}\label{12corDIR2t} Let $D$ be a bounded domain in
${\Bbb C}$ whose boundary consists of $n\geqslant2$ mutually
disjoint Jordan curves and $\mu:D\to{\Bbb C}$ be a measurable
function with $|\mu(z)|<1$ a.e. such that $K_{\mu}\in L^1_{\rm
loc}(D)$. Suppose that
\begin{equation}\label{edgddmultiweIR6**c}
k_{z_{0}}(\varepsilon)=O\left(\left[\log\frac{1}{\varepsilon}\cdot\log\log\frac{1}{\varepsilon}\cdot\ldots\cdot\log\ldots\log\frac{1}{\varepsilon}
\right]\right) \qquad\forall\ z_0\in \overline{D}
\end{equation}
where $k_{z_0}(\varepsilon)$ is the average of the function
$K^T_{\mu}(z,z_0)$ over $S(z_0,\varepsilon)$. Then the Beltrami
equation (\ref{eqBeltrami}) has a multi-valued solutions of the
Dirichlet problem (\ref{eqGrUsl}) for each continuous function
$\varphi:\partial D\to{\Bbb R}$.\end{corollary}

\medskip

\begin{theorem}\label{12thKR4.1} {\it Let $D$ be a bounded domain in
${\Bbb C}$ whose boundary consists of $n\geqslant2$ mutually disjoint Jordan curves and $\mu:D\to{\Bbb C}$ be a
measurable function with $|\mu(z)|<1$ a.e. such that $K_{\mu}\in L^1_{\rm loc}(D)$. Suppose that
\begin{equation}\label{eqKR4.1}\int\limits_{D}\Phi_{z_0}\left(K^T_{\mu}(z,z_0)\right)\,dm(z)<\infty
\qquad\forall\ z_0\in \overline{D} \end{equation} for a convex
non-decreasing function $\Phi_{z_0}:[0,\infty]\to[0,\infty]$ such
that
\begin{equation}\label{eqKR4.2}
\int\limits_{\delta(z_0)}^{\infty}\frac{d\tau}{\tau\Phi^{-1}(\tau)}=\infty\end{equation} for some
$\delta(z_0)>\Phi_{z_0}(0)$. Then the Beltrami equation (\ref{eqBeltrami}) has a multi-valued solutions of the
Dirichlet problem (\ref{eqGrUsl}) for each continuous function $\varphi:\partial D\to{\Bbb R}$.} \end{theorem}

\medskip

\begin{corollary}\label{12corDIR1000} In particular, the conclusion of
Theorem \ref{12thKR4.1} holds if
\begin{equation}\label{eqp.KR4.1c}\int\limits_{D\cap
U_{z_0}}e^{\alpha(z_0) K^T_{\mu}(z,z_0)}\,dm(z)<\infty \qquad\forall\ z_0\in \overline{D}\end{equation} for some
$\alpha(z_0)>0$ and a neighborhood $U_{z_0}$ of the point $z_0$.
\end{corollary}

\medskip

\begin{corollary}\label{corc.12thKR4.1} Let $D$ be a bounded domain in
${\Bbb C}$ whose boundary consists of $n\geqslant2$ mutually
disjoint Jordan curves and $\mu:D\to{\Bbb C}$ be a measurable
function such that $|\mu(z)|<1$ a.e. and
\begin{equation}\label{eqc.KR4.1}\int\limits_{D}\Phi\left(K_{\mu}(z)\right)\,dm(z)<\infty\end{equation}
for a convex non-decreasing function $\Phi:[0,\infty]\to[0,\infty]$. If \begin{equation}\label{eqc.KR4.2}
\int\limits_{\delta}^{\infty}\frac{d\tau}{\tau\Phi^{-1}(\tau)}=\infty\end{equation} for some $\delta>\Phi(0)$.
Then the Beltrami equation (\ref{eqBeltrami}) has a multi-valued solutions of the Dirichlet problem
(\ref{eqGrUsl}) for each continuous function $\varphi:\partial D\to{\Bbb R}$. \end{corollary}

\end{document}